\documentclass{article}
\usepackage{amsmath,amssymb,amsthm}
\usepackage{accents}
\usepackage{color}
\usepackage[top=25truemm,bottom=25truemm,left=25truemm,right=25truemm]{geometry}
\usepackage{graphicx}
\usepackage{hyperref}
\usepackage{mathtools}
\usepackage{url}
\theoremstyle{definition}
\newtheorem{Def}{Definition}[section]
\newtheorem{Lem}[Def]{Lemma}
\newtheorem{Prop}[Def]{Proposition}
\newtheorem{Thm}[Def]{Theorem}
\newtheorem{Cor}[Def]{Corollary}
\newtheorem{Exp}[Def]{Example}
\newtheorem{Rmk}[Def]{Remark}

\newtheorem{Ex}[Def]{Expectation}

\makeatletter
\def\bar{\accentset{{\cc@style\underline{\mskip13mu}}}}

\makeatother

\title{{\bf\textsf{Superspecial plane quintics with large automorphism groups}}}
\author{Ryo Ohashi}

\begin{document}
\maketitle
\begin{abstract}
In this paper, we study plane quintic curves whose automorphism groups have order greater than 10, as well as those with cyclic automorphism groups of order 8 and 10.
The latter two cases are represented as one-parameter families, where their superspeciality can be explicitly described in terms of a truncation of certain Gaussian hypergeometric series.
Applying this characterization, we determine the exact number of isomorphism classes of superspecial plane quintic curves with automorphism groups $\cong \mathbb{Z}/10\mathbb{Z}$.
We also provide an efficient algorithm to enumerate such curves with automorphism groups $\cong \mathbb{Z}/8\mathbb{Z}$, and provide the computational results for the range $13 < p < 10000$.
\end{abstract}

\section{Introduction}
Throughout this paper, a \emph{curve} always means a non-singular projective variety of dimension one defined over an algebraically closed field of characteristic $p > 0$.
A curve is called \emph{superspecial} if its Jacobian is isomorphic to a product of supersingular elliptic curves.
Superspecial curves remain an important topic in number theory\\ and algebraic geometry, with many open problems regarding their existence and the number of isomorphism classes.
In recent years, they have also attracted attention because of their applications in cryptography and coding theory, since they often have many rational points relative to their genus.

It is well known that every plane quartic curve has genus $3$ and is non-hyperelliptic, and conversely, every non-hyperelliptic curve of genus $3$ can be realized as a plane quartic curve.
Several studies have investigated the superspeciality of plane quartic curves.
Oort~\cite[Theorem 5.12(1)]{Oort} showed that there exists a superspecial plane quartic curve in any characteristic $p \geq 3$.
His proof focused on a special family of plane quartic curves called \hspace{-0.2mm}\emph{Ciani curves}, whose Jacobians decompose as products of three elliptic curves.
Also, Moriya-Kudo~\cite{MK} proposed an algorithm for enumerating superspecial Ciani curves.
Another characterization of a Ciani curve is as a plane quartic curve with an automorphism group containing a subgroup isomorphic to $(\mathbb{Z}/2\mathbb{Z})^2$.
The possible automorphism groups of plane quartic curves in characteristic $p \geq 5$ are classified into $13$ types (the explicit list of these groups is found in \cite[Theorem 6.5.2]{Dolgachev}).
Building on this classification, Brock~\cite{Brock} studied superspecial plane quartic curves for each of these types. % using the result by Hashimoto~\cite{Hashimoto} on the class numbers of quaternion hermitian lattices.
In particular, the superspeciality for the types with $0$-dimensional families is determined as follows:\vspace{-0.4mm}
\begin{itemize}
    \item The Klein quartic $X^3Y+Y^3Z+Z^3X = 0$ with an automorphism group of order $168$ is superspecial if and only if $p \equiv 3,5,6 \pmod{7}$.\vspace{-0.8mm}
    \item The Fermat quartic $X^4+Y^4+Z^4 = 0$ with an automorphism group of order $96$ is superspecial if and only if $p \equiv 3 \pmod{4}$.\vspace{-0.8mm}
    \item The plane quartic $X^4+Y^4+YZ^3 = 0$ with an automorphism group of order $48$ is superspecial if and only if $p \equiv 11 \pmod{12}$.\vspace{-0.8mm}
    \item The plane quartic $X^4+XY^3+YZ^3 = 0$ with cyclic automorphism group of order $9$ is superspecial if and only if $p \equiv 8 \pmod{9}$.\vspace{0.4mm}
\end{itemize}
Also, there are $3$ types which admit 1-dimensional families.
Among them, those whose automorphism groups have order $24$ or $16$ are Ciani curves; \hspace{0.5mm}their superspeciality can be completely described by the supersingularity of elliptic curves (the explicit equations of them are found in \cite[Sections 4--5]{MS}).
The remaining type consists of curves with cyclic automorphism groups of order $6$, for which it is known (cf. \cite{OKH}) that there exist exactly $\lfloor p/12 \rfloor$ superspecial curves if $p \equiv 5 \pmod{6}$, whereas no such curves exist if $p \equiv 1 \pmod{6}$.

This paper focuses on \emph{plane quintic curves}, all of which have genus $6$ and are non-hyperelliptic.
Note that, in contrast to plane quartic curves, there exist non-hyperelliptic curves of genus $6$ that are not isomorphic to any plane quintic curve (cf. \cite[Example\hspace{1mm}IV.5.6]{Hartshorne}).
Badr-Bars~\cite{BB} classified the possible automorphism groups of plane quintic curves in characteristic $p > 13$ into $14$ types, and gave an explicit defining equation for each type.
Among these $14$ types, those with $0$-dimensional families are as follows:
\begin{itemize}
\item The Fermat quintic $X^5+Y^5+Z^5 = 0$ with an automorphism group of order $150$.\vspace{-1.2mm}
\item The Hurwitz quintic $X^4Y+Y^4Z+Z^4X = 0$ with an automorphism group of order $39$.\vspace{-1.2mm}
\item The plane quintic $X^5+Y^4Z+YZ^4 = 0$ with an automorphism group of order $30$.\vspace{-1.2mm}
\item The plane quintic $X^5+Y^5+XZ^4 = 0$ with cyclic automorphism group of order $20$.\vspace{-1.2mm}
\item The plane quintic $X^5+Y^4Z+XZ^4 = 0$ with cyclic automorphism group of order $16$.\vspace{0.6mm}
\end{itemize}
In Section \ref{subsec:type1-5}, we provide explicit necessary and sufficient conditions on $p$ for these curves to be superspecial.
In addition, there are $2$ types which admit $1$-dimensional families.
One of them is the family of plane quintics defined by the equation\vspace{-1mm}
\[
    X^5+Y^5+XZ^4+rX^3Z^2 = 0 \ \text{ with }\,r \neq 0,\pm 2,\hspace{0.7mm}r^2 \neq 20\vspace{-0.2mm}
\]
which has a cyclic automorphism group of order $10$.
In Section \ref{subsec:type6}, we investigate several properties of such a curve.
In particular, we show that its superspeciality can be described by (a truncation of) a single Gaussian hypergeometric series, which leads to our first main theorem:\vspace{-0.8mm}
\begin{Thm}\label{thm:main-10}
The number of isomorphism classes of superspecial plane quintic curves with cyclic automorphism group of order $10$ equals\vspace{-0.4mm}
\[
    \left\{ 
        \begin{array}{ll}
            (3p-27)/20 & \text{ if } p \equiv 9 \!\!\pmod{20},\\[-0.2mm]
            (3p-37)/20 & \text{ if } p \equiv 19 \!\!\pmod{20},\\[-0.2mm]
            0 & \text{ otherwise}
        \end{array}
    \right.\vspace{-1.2mm}
\]
in characteristic $p > 13$.\vspace{-0.8mm}
\end{Thm}
\noindent In Section \ref{subsec:type8}, we study another 1-dimensional family of plane quintics by the equation\vspace{-0.5mm}
\[
    X^5+Y^4Z+XZ^4+rX^3Z^2 = 0 \ \text{ with }\,r \neq 0,\pm 2,\vspace{0.5mm}
\]
which has a cyclic automorphism group of order $8$; \hspace{0.5mm}however, unlike the case of order $10$, its superspeciality is characterized by multiple Gaussian hypergeometric series.
Consequently, the number of isomorphism classes of such superspecial curves cannot be expressed by a simple explicit formula like Theorem \ref{thm:main-10}.
We therefore propose an efficient algorithm (Theorem \ref{prethm:main-8}) to determine this number in small characteristic $p$.
Executing our algorithm in the range $13 < p < 10000$, we obtain our second main theorem:\vspace{-0.8mm}
\begin{Thm}\label{thm:main-8}
For each characteristic $p$ in the range $13 < p < 10000$, there are no superspecial plane quintic curves with a cyclic automorphism group of order $8$ if $p \equiv 1,3,5 \pmod{8}$.
Also, the numbers of isomorphism classes of such curves for $p \equiv 7 \pmod{8}$ are listed in Table \ref{tbl:main}.
\end{Thm}
\noindent Our implementation code is available at the following URL:\vspace{0.5mm}
\begin{center}
    \url{https://github.com/Ryo-Ohashi/SSpZ8quintic}.\vspace{1.3mm}
\end{center}
We remark that it took only approximately 36.2 seconds to obtain the result of Theorem \ref{thm:main-8} (see Section \ref{subsec:type8} for details on the computing environment), which suggests that the upper bound for $p$ can be easily updated.
The proof of (non)-existence and the number of isomorphism classes of superspecial plane quintic curves with a cyclic automorphism group of order $8$ for general characteristics is left for future work.\vspace{-2.8mm}

\paragraph{Organization.} The remainder of this paper is organized as follows.
In Section \ref{sec:preliminaries}, we review the relationship between the superspeciality of specific superelliptic curves and Gaussian hypergeometric series.
In Section \ref{sec:main}, we discuss the superspeciality of plane quintic curves with automorphism groups of order greater than $10$, or cyclic automorphism groups of order $8$ and $10$.
Finally, we give concluding remarks in Section \ref{sec:concluding}.\vspace{-2.8mm}

\paragraph{Acknowledgements.} This research was supported by JSPS Grant-in-Aid for Young Scientists 25K17225.

\newpage
\section{Preliminaries}\label{sec:preliminaries}
In this section, we recall several properties of superelliptic curves that will be used in the subsequent section.
Let $K$ be an algebraically closed field of characteristic $p > 2$.
A \emph{superelliptic curve} over $K$ is a curve defined by a projective equation\vspace{-0.7mm}
\begin{equation}\label{eq:superelliptic}
    Y^nZ^{m-n} = F(X,Z)\ \text{ with }\,n \hspace{-0.1mm}\geq 2 \,\text{ and }\, p \nmid n,\vspace{0.4mm}
\end{equation}
where $F(X,Z) \in K[X,Z]$ is a binary form of degree $m \geq 3$ with no multiple factors.
% A superelliptic curve in the case where $n=2$ is no other than a \emph{hyperelliptic curve}.
The Riemann-Hurwitz formula tells us that the genus of the superelliptic curve \eqref{eq:superelliptic} is given by\vspace{-1.2mm}
\[
    g = \frac{(m-1)(n-1)-{\rm gcd}(m,n)+1}{2}.
\]
Also, the following is a fundamental proposition:\vspace{-1.2mm}
\begin{Prop}[{cf. \cite[Section 8]{MS}}]\label{prop:isomorphic}
The superelliptic curves $Y^nZ^{m-n} = F(X,Z)$ and $Y^nZ^{m-n} = G(X,Z)$ over $K$ are isomorphic to each other if and only if $F(X,Z)$ and $G(X,Z)$ are equivalent under ${\rm GL}_2(K)$.
\end{Prop}

In what follows, by setting $x = X/Z$ and $y = Y/Z$, we identify the superelliptic curve \eqref{eq:superelliptic} with its affine model $y^n = f(x)$, where $f(x) \coloneqq F(x,1) \in K[x]$ is a square-free polynomial of degree $m \geq 3$.
Let us describe a necessary and sufficient condition for superelliptic curves to be superspecial:\vspace{-1.2mm}
\begin{Thm}[{cf. \cite[Theorem 2.15]{Brock}}]\label{thm:superelliptic}
The superelliptic curve given by $y^n = f(x)$ is superspecial if and only if the $x^{hp-k}$-coefficients in $f(x)^{(ip-j)/n}$ are equal to $0$ for all integers $(i,j,h,k)$ satisfying\vspace{-1.3mm}
\[
    1 \leq i,j < n, \quad 1 \leq h < mi/n, \quad 1 \leq k < mj/n,\ \text{ and }\ n \mid\hspace{-0.2mm} (ip-j),\vspace{-0.1mm}
\]
where $m \geq 3$ is the degree of $f(x)$.
\end{Thm}
% \begin{Exp}
% Consider a hyperelliptic curve $H: y^2 = f(x)$ of genus $g$, where $f(x) \in K[x]$ is a square-free polynomial of degree $2g+1$ or $2g+2$.
% In this case, the pair $(i,j)$ satisfying the conditions in Theorem \ref{thm:superelliptic} is only $(i,j) = (1,1)$.
% Therefore, the hyperelliptic curve $H$ is superspecial if and only if all the $x^{hp-k}$-coefficients in $f(x)^{(p-1)/2}$ are equal to $0$ for $1 \leq h,k \leq g$.
% This is precisely the result of \cite[Section 2]{Yui}.
% \end{Exp}

As a corollary of Theorem \ref{thm:superelliptic}, the superelliptic curve $y^n = x^m-x$ is superspecial if and only if\vspace{-0.6mm}
\begin{equation}\label{eq:S}
    S \coloneqq \left\{
    \,(i,j,h,k) \in \mathbb{Z}^4\ \middle|
    \begin{array}{l}
        1 \leq i,j < n,\ n \mid (ip-j),\\[-0.2mm]
        1 \leq h < mi/n,\ 1 \leq k < mj/n,\ (m-1) \mid \bigl(hp-k - \frac{ip-j}{n}\bigr),\\
        0 \leq hp-k - \frac{ip-j}{n} \leq (m-1)\frac{ip-j}{n}
    \end{array}\hspace{-1mm}\right\}\vspace{-0.1mm}
\end{equation}
is empty.
Indeed, it follows from the binomial theorem that\vspace{-1mm}
\begin{align*}
    (x^m-x)^{(ip-j)/n} &= x^{(ip-j)/n}(x^{m-1}-1)^{(ip-j)/n}\\[-0.9mm]
    &= \sum_{\ell=0}^{(ip-j)/n}\binom{\frac{ip-j}{n}}{\ell}(-1)^{(ip-j)/n-\ell}x^{(ip-j)/n+(m-1)\ell}.\\[-6.1mm]
\end{align*}
Since $0 \leq (ip-j)/n < p$, one has\vspace{-2.2mm}
\[
    \binom{\frac{ip-j}{n}}{\ell}(-1)^{(ip-j)/n-\ell} \not\equiv 0 \pmod{p}\vspace{-0.3mm}
\]
for every $0 \leq \ell \leq (ip-j)/n$.
Hence, the $x^{hp-k}$-coefficient in $(x^m-x)^{(ip-j)/n}$ is non-zero if and only if there exists an integer $\ell$ such that\vspace{-1.2mm}
\[
    hp-k = (ip-j)/n + (m-1)\ell \ \text{ with }\ 0 \leq \ell \leq (ip-j)/n.\vspace{0.4mm}
\]
This condition is equivalent to\vspace{-1.3mm}
\[
    (m-1) \mid \Bigl(hp-k - \frac{ip-j}{n}\Bigr) \ \text{ and }\ 0 \leq hp-k - \frac{ip-j}{n} \leq (m-1)\frac{ip-j}{n},
\]
as desired.
Here, the following sufficient conditions for $S$ to be empty were given in \cite[Corollary 2.16]{Brock}.\vspace{-1.2mm}
\begin{Cor}[{cf. \cite[Corollary 2.16]{Brock}}]\label{cor:x^d-x}
The set $S$ in \eqref{eq:S} is empty if one of the following conditions holds:\vspace{-1.1mm}
\begin{enumerate}
\item[(i)] $p \equiv -1 \pmod{(m-1)n}$, or\vspace{-2.7mm}
\item[(ii)] $p \equiv m \pmod{(m-1)n}$ when $n$ is a divisor of $m+1$.\vspace{-0.6mm}
\end{enumerate}
Consequently, the superelliptic curve $y^n = x^m - x$ is superspecial in each of these cases.
\end{Cor}

Next, we fix a pair $(n,r)$ of positive integers where $n \geq 2$ and $p \nmid n$. 
Consider the following 1-dimensional family of superelliptic curves defined by the equation\vspace{-1mm}
\begin{equation}\label{eq:Clambda}
    C_\lambda: y^n = x(x^r-1)(x^r-\lambda) \ \text{ with }\, \lambda \neq 0,1,\vspace{-0.5mm}
\end{equation}
and the following set\vspace{-0.8mm}
\begin{equation}\label{eq:T}
    T \coloneqq \left\{
    \,(i,j,h,k) \in \mathbb{Z}^4\ \middle|
    \begin{array}{l}
        1 \leq i,j < n,\ n \mid (ip-j),\\[-0.2mm]
        1 \leq h < (2r+1)i/n,\ 1 \leq k < (2r+1)j/n,\ r \mid \bigl(hp-k - \frac{ip-j}{n}\bigr),\\
        0 \leq hp-k-\frac{ip-j}{n} \leq 2r\frac{ip-j}{n}
    \end{array}\hspace{-1mm}\right\}.\vspace{-0.4mm}
\end{equation}
Then, the superspeciality of $C_\lambda$ can be described in terms of Gaussian hypergeometric series as follows:\vspace{-1.3mm}
\begin{Thm}\label{thm:Clambda}
The superelliptic curve $C_\lambda$ in \eqref{eq:Clambda} is superspecial if and only if the following conditions hold for every $4$-tuple $(i,j,h,k) \in T$, where $T$ is the set defined in \eqref{eq:T}:\vspace{-1.1mm}
\begin{itemize}
    \item If $hp-k-\frac{ip-j}{n} \leq r\frac{ip-j}{n}$, then\vspace{-1.2mm}
    \[
        G^{((hp-k-\frac{ip-j}{n})/r)}\scalebox{0.95}{$\displaystyle\left(\frac{j}{n},\hspace{0.3mm}\frac{k}{r}-\frac{j}{nr},\hspace{0.5mm}1-\frac{j}{n}+\frac{k}{r}-\frac{j}{nr} \ \middle|\  \lambda\right)$} = 0.\vspace{-1.5mm}
    \]
    \item If $hp-k-\frac{ip-j}{n} > r\frac{ip-j}{n}$, then\vspace{-0.6mm}
    \[
        G^{(2\frac{ip-j}{n}-(hp-k-\frac{ip-j}{n})/r)}\scalebox{0.95}{$\displaystyle\left(\frac{j}{n},-\frac{k}{r}+\frac{2j}{n}+\frac{j}{nr},\hspace{0.5mm}1+\frac{j}{n}-\frac{k}{r}+\frac{j}{nr} \ \middle|\  \lambda\right)$} = 0.\vspace{-2.2mm}
    \]
\end{itemize}
Here, $G^{(d)}(a,b,c \mid z)$ denotes the truncated \hspace{-0.1mm}Gaussian hypergeometric series defined in \cite[Definition 2.1.2]{OK}, namely, we write\vspace{-1.1mm}
\begin{equation}\label{eq:Gaussian}
    G^{(d)}(a,b,c \mid z) \coloneqq \sum_{\ell=0}^d \frac{(a\,;\ell)(b\,;\ell)}{(c\,;\ell)(1\,;\ell)}z^\ell\vspace{1.7mm}
\end{equation}
where $(x\,;0) = 1$ and $(x\,;\ell) = x(x+1) \cdots (x+\ell-1)$ for an integer $\ell \geq 1$.\vspace{-1.1mm}
\end{Thm}
\begin{proof}
It follows from Theorem \ref{thm:superelliptic} that $C_\lambda$ is superspecial if and only if, for all integers $(i,j,h,k)$ satisfying the conditions\vspace{-1mm}
\begin{equation}\label{eq:condition}
    1 \leq i,j < n, \quad 1 \leq h < (2r+1)i/n, \quad 1 \leq k < (2r+1)j/n,\ \text{ and }\ n \mid\hspace{-0.2mm} (ip-j),\vspace{0.8mm}
\end{equation}
the $x^{hp-k}$-coefficients in $\{x(x^r-1)(x^r-\lambda)\}^{(ip-j)/n}$ are equal to $0$.
Since\vspace{-0.6mm}
\begin{align*}
    (x^r-1)^{(ip-j)/n} &= \sum_{\ell_1=0}^{(ip-j)/n}\binom{(ip-j)/n}{\ell_1}(-1)^{(ip-j)/n-\ell_1}x^{r\ell_1}, \,\text{ and}\\[-0.3mm]
    (x^r-\lambda)^{(ip-j)/n} &= \sum_{\ell_2=0}^{(ip-j)/n}\binom{(ip-j)/n}{\ell_2}(-\lambda)^{(ip-j)/n-\ell_2}x^{r\ell_2}\\[-5.3mm]
\end{align*}
by the binomial theorem, one can expand $\{x(x^r-1)(x^r-\lambda)\}^{(ip-j)/n}$ as\vspace{-1.5mm}
\begin{equation}\label{eq:expand}
    \sum_{\ell=0}^{2(ip-j)/n}\hspace{-0.4mm}\Biggl(\hspace{0.5mm}\sum_{\ell_1+\ell_2=\ell}\binom{(ip-j)/n}{\ell_1}\!\binom{(ip-j)/n}{\ell_2}(-1)^{(ip-j)/n-\ell_1}(-\lambda)^{(ip-j)/n-\ell_2}\hspace{-0.5mm}\Biggr)x^{(ip-j)/n+r\ell}.\vspace{-0.3mm}
\end{equation}
Hence, for integers $(i,j,h,k)$ satisfying \eqref{eq:condition}, the $x^{hp-k}$-coefficient in $\{x(x^r-1)(x^r-\lambda)\}^{(ip-j)/n}$ is non-zero \emph{only if} there exists an integer $\ell$ such that\vspace{-0.7mm}
\[
    hp-k = (ip-j)/n + r\ell \ \text{ with }\ 0 \leq \ell \leq 2(ip-j)/n,\vspace{-0.3mm}
\]
which is equivalent to\vspace{-1.7mm}
\[
    r \mid \Bigl(hp-k-\frac{ip-j}{n}\Bigr) \text{ and } 0 \leq hp-k-\frac{ip-j}{n} \leq 2r\frac{ip-j}{n}.\vspace{0.4mm}
\]
This is precisely the condition appearing in the definition of $T$.

\newpage
For each $(i,j,h,k) \in T$, up to a sign, the $x^{hp-k}$-coefficient in \eqref{eq:expand} is given by\vspace{-1.1mm}
\[
    \sum_{\ell_1+\ell_2=(hp-k-\frac{ip-j}{n})/r}\!\binom{(ip-j)/n}{\ell_1}\!\binom{(ip-j)/n}{\ell_2}\lambda^{(ip-j)/n-\ell_2} = \sum_{\ell_1=\max(0,b-a)}^{\min(a,b)}\!\binom{a}{\ell_1}\!\binom{a}{b-\ell_1}\lambda^{a-b+\ell_1},\vspace{-0.6mm}
\]
where we define\vspace{-1.4mm}
\[
    a \coloneqq \frac{ip-j}{n}\ \text{ and }\ b \coloneqq \frac{hp-k-\frac{ip-j}{n}}{r}.\vspace{0.8mm}
\]
We divide the proof into two cases according to whether $b \leq a$ or $b > a$:\vspace{-1.1mm}
\begin{itemize}
    \item If $b \leq a$, then a straightforward computation (cf. \cite[Lemma A.1]{OK}) leads to\vspace{-1mm}
    \[
        \sum_{\ell_1=0}^b\!\binom{a}{\ell_1}\!\binom{a}{b-\ell_1}\lambda^{a-b+\ell_1} = \lambda^{a-b} \cdot \binom{a}{b}G^{(b)}(-a,-b,1+a-b \mid \lambda).\vspace{-1.1mm}
    \]
    Also, we have\vspace{-0.7mm}
    \[
        -a \equiv \frac{j}{n}, \quad -b \equiv \frac{k}{r} - \frac{j}{nr}, \ \text{ and }\ 1+a-b \equiv 1-\frac{j}{n}+\frac{k}{r} - \frac{j}{nr}\vspace{0.9mm}
    \]
    modulo $p$, and thus the $x^{hp-k}$-coefficient in \eqref{eq:expand} vanishes if and only if\vspace{-0.6mm}
    \[
        G^{(b)}\scalebox{0.95}{$\displaystyle\left(\frac{j}{n},\hspace{0.3mm}\frac{k}{r}-\frac{j}{nr},\hspace{0.5mm}1-\frac{j}{n}+\frac{k}{r}-\frac{j}{nr}\ \middle|\ \lambda\right)$} = 0\vspace{-0.4mm}
    \]
    since $b \leq a < p$ and $\lambda \neq 0$.\vspace{-1.3mm}
    \item If $b > a$, then a straightforward computation (cf. \cite[Lemma A.1]{OK}) leads to\vspace{-1mm}
    \[
        \sum_{\ell_1=b-a}^a\!\binom{a}{\ell_1}\!\binom{a}{b-\ell_1}\lambda^{a-b+\ell_1} = \binom{a}{2a-b}G^{(2a-b)}(-a,b-2a,1+b-a \mid \lambda).\vspace{-0.8mm}
    \]
    Also, we have\vspace{-2.1mm}
    \[
        -a \equiv \frac{j}{n}, \quad b-2a \equiv -\frac{k}{r}+\frac{2j}{n}+\frac{j}{nr}, \ \text{ and }\ 1+b-a \equiv 1+\frac{j}{n}-\frac{k}{r}+\frac{j}{nr}
    \]
    modulo $p$, and thus the $x^{hp-k}$-coefficient in \eqref{eq:expand} vanishes if and only if\vspace{-0.6mm}
    \[
        G^{(2a-b)}\scalebox{0.95}{$\displaystyle\left(\frac{j}{n},-\frac{k}{r}+\frac{2j}{n}+\frac{j}{nr},\hspace{0.5mm}1+\frac{j}{n}-\frac{k}{r}+\frac{j}{nr} \ \middle|\  \lambda\right)$} = 0\vspace{-0.9mm}
    \]
    since $2a-b < a < p$.\vspace{-0.6mm}
\end{itemize}
Combining these two cases, we obtain the desired assertion.\vspace{-1mm}
\end{proof}
\begin{Exp}
Let $C_\lambda$ be a genus-$g$ hyperelliptic curve defined by the equation\vspace{-0.5mm}
\[
    C_\lambda: y^2 = x(x^g-1)(x^g-\lambda)\ \text{ with }\, \lambda \neq 0,1.\vspace{0.8mm}
\]
In this case, the set $T$ in \eqref{eq:T} is given by\vspace{-0.6mm}
\[
    T = \left\{
    \,(1,1,h,k) \in \mathbb{Z}^4\ \middle|
    \begin{array}{l}
        1 \leq h,\hspace{-0.1mm}k \leq g,\ g \mid\hspace{-0.2mm} \bigl(hp-k -\hspace{-0.2mm} \frac{p-1}{2}\bigr),\\
        0 \leq hp-k-\hspace{-0.2mm}\frac{p-1}{2} \hspace{-0.2mm}\leq g(p-1)
    \end{array}\hspace{-1mm}\right\}.
\]
Applying Theorem \ref{thm:Clambda}, the hyperelliptic curve $C_\lambda$ is superspecial if and only if the following conditions hold:\vspace{-0.8mm}
\begin{itemize}
    \item If $hp-k-\hspace{-0.2mm}\frac{p-1}{2} \hspace{-0.2mm}\leq \frac{g(p-1)}{2}$, then\vspace{-1.5mm}
    \[
        G^{((hp-k-\frac{p-1}{2})/g)}\scalebox{0.95}{$\displaystyle\left(\frac{1}{2},\hspace{0.3mm}\frac{k}{g}-\frac{1}{2g},\hspace{0.3mm}\frac{1}{2}+\frac{k}{g}-\frac{1}{2g} \ \middle|\  \lambda\right)$} = 0.\vspace{-1.6mm}
    \]
    \item If $hp-k-\hspace{-0.2mm}\frac{p-1}{2} \hspace{-0.2mm}> \frac{g(p-1)}{2}$, then\vspace{-0.6mm}
    \[
        G^{(p-1-(hp-k-\frac{p-1}{2})/g)}\scalebox{0.95}{$\displaystyle\left(\frac{1}{2},\hspace{0.3mm}1-\frac{k}{g}+\frac{1}{2g},\hspace{0.3mm}\frac{3}{2}-\frac{k}{g}+\frac{1}{2g} \ \middle|\  \lambda\right)$} = 0.\vspace{-0.8mm}
    \]
\end{itemize}
for every $(1,1,h,k) \in T$.\vspace{-3mm}
\end{Exp}

\newpage
\section{Main results}\label{sec:main}
In the following, let $K$ be an algebraically closed field of characteristic $p > 13$, and we consider plane quintic curves over $K$, namely,\vspace{-0.4mm}
\[
    C: F(X,Y,Z) = 0 \quad \begin{array}{l}\text{\hspace{-0.2mm}where $F(X,Y,Z) \in K[X,Y,Z]$ is an irreducible}\\[-0.2mm]\text{homogeneous polynomial of degree 5},\end{array}\vspace{0.5mm}
\]
which are non-hyperelliptic curves of genus $6$.
Badr and Bars~\cite{BB} classified all the plane quintic curves over $K$ into $14$ types according to their automorphism groups, and gave an explicit defining equation for each type.
Table~\ref{tbl:classification} below lists the types whose families have dimension at most $1$ (we denote by $\mathbf{G}_n$ a group of order $n$ for a positive integer $n$).\vspace{1mm}

\begin{table}[htbp]
    \centering
    \begin{tabular}{c||c|c|c}\hline
       Type & ${\rm Aut}(C)$ & Defining equation of $C$ & Conditions\\\hline
       \textit{1} & $\mathbf{G}_{150}$ & $X^5+Y^5+Z^5 = 0$ & ------\\[-0.2mm]
       \textit{2} & $\mathbf{G}_{39}$ & $X^4Y+Y^4Z+Z^4X = 0$ & ------\\[-0.2mm]
       \textit{3} & $\mathbf{G}_{30}$ & $X^5+Y^4Z+YZ^4 = 0$ & ------\\[-0.2mm]
       \textit{4} & $\mathbb{Z}/20\mathbb{Z}$ & $X^5+Y^5+XZ^4 = 0$ & ------\\[-0.2mm]
       \textit{5} & $\mathbb{Z}/16\mathbb{Z}$ & $X^5+Y^4Z+XZ^4 = 0$ & ------\\[-0.2mm]
       \textit{6} & $\mathbb{Z}/10\mathbb{Z}$ & $X^5+Y^5+XZ^4+rX^3Z^2 = 0$ & $r \neq 0,\pm 2, \hspace{0.5mm}r^2 \neq 20$\\[-0.2mm]
       \textit{8} & $\mathbb{Z}/8\mathbb{Z}$ & $X^5+Y^4Z+XZ^4+rX^3Z^2 = 0$ & $r \neq 0,\pm 2$\\\hline
    \end{tabular}\vspace{-0.7mm}
    \caption{Automorphism groups of plane quintic curves with families of dimension $\leq 1$}\label{tbl:classification}\vspace{0.7mm}
\end{table}
\noindent In Section \ref{subsec:type1-5}, we give necessary and sufficient conditions on $p$ for the plane quintic curves corresponding to the $0$-dimensional families (i.e., Types \hspace{-0.2mm}\textit{1}--\textit{5}\hspace{-0.2mm}) to be superspecial.
In Section \ref{subsec:type6} \hspace{0.2mm}(resp. Section \ref{subsec:type8}), we discuss the superspeciality of plane quintic curves of Type \hspace{-0.2mm}\textit{6} (resp. \hspace{-0.4mm}Type \hspace{-0.2mm}\textit{8}\hspace{-0.2mm}).

\setcounter{equation}{0}
\subsection{Plane quintic curves with 0-dimensional families}\label{subsec:type1-5}\vspace{0.5mm}
First, we determine the characteristics $p$ for which plane quintic curves of Types \hspace{-0.2mm}\textit{1}--\textit{2} are superspecial.\vspace{-1mm}
\begin{Prop}\label{prop:type1-2}
The following statements hold:\vspace{-0.6mm}
\begin{enumerate}
\item[(1)] The Fermat quintic $X^5+Y^5+Z^5 = 0$ is superspecial if and only if $p \equiv 4 \pmod{5}$.\vspace{-2.6mm}
\item[(2)] The Hurwitz quintic $X^4Y+Y^4Z+Z^4X = 0$ is superspecial if and only if $p \equiv 4,10,12 \pmod{13}$.\vspace{0.3mm}
\end{enumerate}
\end{Prop}
\begin{proof}
(1) This is a direct consequence of \cite[Corollary 1]{KW} applied to the case $m=5$.\par
\hspace{5.8mm}(2) Applying \cite[Proposition 4.1]{MS} to the case $n=4$, we see that this curve is superspecial if and only if, for every $(i,j) \in\hspace{-0.2mm} \{(0,0),(0,1),(1,0),(0,2),(1,1),(2,0)\} \eqqcolon U$, the system of congruences
\begin{equation}\label{eq:Case2}
    \left\{\,
    \begin{aligned}
        4k-h+i &\equiv 0,\\[-1.2mm]
        4(h-k)+k+j &\equiv p-1
    \end{aligned}
    \right. \pmod{p}\vspace{-0.2mm}
\end{equation}
has no integer solution $(h,k)$ satisfying the inequality $0 \leq k \leq h \leq p-1$.
By solving \eqref{eq:Case2} for each $(i,j) \in U$, we obtain\vspace{-0.7mm}
\[
    13h \equiv -3i-4j-4 \ \,\text{ and }\ 13k \equiv -4i-j-1 \pmod{p}.\vspace{1.2mm}
\]
Therefore, for some integers $s$ and $t$, we can write\vspace{-0.9mm}
\[
    h = \frac{sp-(3i+4j+4)}{13} \ \,\text{ and }\ k = \frac{tp-(4i+j+1)}{13}.\vspace{0.8mm}
\]
The integer solution $(h,k)$ satisfies $0 \leq h,k \leq p-1 $ if and only if $s,t \in \{1,\dots,12\}$. Equivalently, $s$ and $t$ are the smallest positive integers satisfying\vspace{-0.4mm}
\[
    \left\{\,
    \begin{aligned}
        sp &\equiv 3i+4j+4,\\[-1.2mm]
        tp &\equiv 4i+j+1
    \end{aligned}
    \right. \pmod{13}.\vspace{-0.1mm}
\]
Thus, these values are determined solely by $(i,j)$ and $p \bmod\hspace{-0.2mm} 13$.
Table \ref{tbl:st} below gives the corresponding values of $(s,t)$ for each $(i,j) \in U$ and $p \bmod{13}$.

\newpage
\begin{table}[tbp]\vspace{-6mm}
    \centering
    \rotatebox{90}{$\hspace{-12mm}p \bmod 13$}\hspace{-1mm}
    \begin{tabular}{c|cccccc}
       \multicolumn{1}{c}{} & \multicolumn{6}{c}{$(i,j)$}\\[2mm]
       & $(0,0)$ & $(0,1)$ & $(1,0)$ & $(0,2)$ & $(1,1)$ & $(2,0)$\\\hline
       $1$ & $(4,1)$ & $(8,2)$ & $(7,5)$ & $(12,3)$ & $(11,6)$ & $(10,9)$\\
       $2$ & $(2,7)$ & $(4,1)$ & $(10,9)$ & $(6,8)$ & $(12,3)$ & $(5,11)$\\
       $3$ & $(10,9)$ & $(7,5)$ & $(11,6)$ & $(4,1)$ & $(8,2)$ & $(12,3)$\\
       $4$ & $(1,10)$ & $(2,7)$ & $(5,11)$ & $(3,4)$ & $(6,8)$ & $(9,12)$\\
       $5$ & $(6,8)$ & $(12,3)$ & $(4,1)$ & $(5,11)$ & $(10,9)$ & $(2,7)$\\
       $6$ & $(5,11)$ & $(10,9)$ & $(12,3)$ & $(2,7)$ & $(4,1)$ & $(6,8)$\\
       $7$ & $(8,2)$ & $(3,4)$ & $(1,10)$ & $(11,6)$ & $(9,12)$ & $(7,5)$\\
       $8$ & $(7,5)$ & $(1,10)$ & $(9,12)$ & $(8,2)$ & $(3,4)$ & $(11,6)$\\
       $9$ & $(12,3)$ & $(11,6)$ & $(8,2)$ & $(10,9)$ & $(7,5)$ & $(4,1)$\\
       $10$ & $(3,4)$ & $(6,8)$ & $(2,7)$ & $(9,12)$ & $(5,11)$ & $(1,10)$\\
       $11$ & $(11,6)$ & $(9,12)$ & $(3,4)$ & $(7,5)$ & $(1,10)$ & $(8,2)$\\
       $12$ & $(9,12)$ & $(5,11)$ & $(6,8)$ & $(1,10)$ & $(2,7)$ & $(3,4)$
    \end{tabular}\vspace{-1mm}
    \caption{Corresponding values of $(s,t)$ for each $(i,j) \in U$ and $p \bmod{13}$}\label{tbl:st}
\end{table}

Furthermore, the condition $k \leq h$ is equivalent to $(t-s)p \leq i-3j-3$.
For each $(i,j) \in U$, we obtain the following:\vspace{-0.4mm}
\begin{itemize}
    \item If $t < s$, then $(t-s)p < -p < -13 < i-3j-3$.\vspace{-2.4mm}
    \item If $s \leq t$, then $(t-s)p \geq 0 > i-3j-3$.\vspace{0.4mm}
\end{itemize}
This tells us that $k \leq h$ holds if and only if $t < s$.
Consequently, it suffices to determine whether $s \leq t$ holds for all $(i,j) \in U$.
By Table~\ref{tbl:st}, this condition is satisfied if and only if $p \equiv 4, 10, 12 \pmod{13}$.
This completes the proof.
\end{proof}

Next, we determine the characteristics $p$ for which plane quintic curves of Types \hspace{-0.2mm}\textit{3}--\textit{5} are superspecial by using the results obtained in the first half of Section \ref{sec:preliminaries}.\vspace{-1mm}
\begin{Prop}\label{prop:type3-5}
The following statements hold:\vspace{-0.6mm}
\begin{enumerate}
\item[(1)] The plane quintic $X^5+Y^4Z+YZ^4 = 0$ is superspecial if and only if $p \equiv 4 \pmod{5}$.\vspace{-2.6mm}
\item[(2)] The plane quintic $X^5+Y^5+XZ^4 = 0$ is superspecial if and only if $p \equiv 19 \pmod{20}$.\vspace{-2.6mm}
\item[(3)] The plane quintic $X^5+Y^4Z+XZ^4 = 0$ is superspecial if and only if $p \equiv 15 \pmod{16}$.\vspace{0.3mm}
\end{enumerate}
\end{Prop}
\begin{proof}
First of all, we recall that the superelliptic curve $y^n = x^m-x$ is superspecial if and only if the set $S$ defined in \eqref{eq:S} is empty.\par
\hspace{5.8mm}(1) This curve is isomorphic to $y^5 = x^4-x$ via the map $(X:Y:Z) = (-y:-x:1)$.
Hence, it follows from Corollary~\ref{cor:x^d-x} that it is superspecial if $p \equiv 4,14 \pmod{15}$, that is, if $p \equiv 4 \pmod{5}$.
In all other cases, by taking $4$-tuple $(i,j,h,k)$ as\vspace{-0.9mm}
\[
    (i,j,h,k) \coloneqq \left\{
        \begin{array}{ll}
            \!(2,2,1,1) & \text{if } p \equiv 1 \hspace{-2mm}\pmod{5},\\
            \!(2,4,1,2) & \text{if } p \equiv 2 \hspace{-2mm}\pmod{5},\\
            \!(4,2,2,1) & \text{if } p \equiv 3 \hspace{-2mm}\pmod{5},
        \end{array}
    \right.\vspace{0.2mm}
\]
we see that $(i,j,h,k) \in S$.
Therefore, this curve is not superspecial.\par
\hspace{5.8mm}(2) This curve is isomorphic to $y^5 = x^5-x$ via the map $(X:Y:Z) =(\zeta^5x:\zeta y:1)$ where $\zeta$ denotes a primitive $40$th root of unity.
Hence, it follows from Corollary~\ref{cor:x^d-x} that it is superspecial if $p \equiv 19 \pmod{20}$.
In all other cases, by taking $4$-tuple $(i,j,h,k)$ as\vspace{0.5mm}
\[
    (i,j,h,k) \coloneqq \left\{
        \begin{array}{ll}
            \!(2,2,1,1) & \text{if } p \equiv 1 \hspace{-2mm}\pmod{20},\\
            \!(3,4,2,1) & \text{if } p \equiv 3 \hspace{-2mm}\pmod{20},\\
            \!(2,4,1,1) & \text{if } p \equiv 7 \hspace{-2mm}\pmod{20},\\
            \!(2,3,1,2) & \text{if } p \equiv 9 \hspace{-2mm}\pmod{20},\\
            \!(3,3,1,1) & \text{if } p \equiv 11 \hspace{-2mm}\pmod{20},\\
            \!(3,4,1,2) & \text{if } p \equiv 13 \hspace{-2mm}\pmod{20},\\
            \!(2,4,1,3) & \text{if } p \equiv 17 \hspace{-2mm}\pmod{20},
        \end{array}
    \right.\vspace{0.2mm}
\]
we see that $(i,j,h,k) \in S$.
Therefore, this curve is not superspecial.

\newpage
\hspace{5.8mm}(3) This curve is isomorphic to $y^4 = x^5-x$ via the map $(X:Y:Z) =(\zeta^4x:\zeta y:1)$ where $\zeta$ denotes a primitive $32$th root of unity.
Hence, it follows from Corollary~\ref{cor:x^d-x} that it is superspecial if $p \equiv 15 \pmod{16}$.
In all other cases, by taking $4$-tuple $(i,j,h,k)$ as\vspace{0.2mm}
\[
    (i,j,h,k) \coloneqq \left\{
        \begin{array}{ll}
            \!(1,1,1,1) & \text{if } p \equiv 1 \hspace{-2mm}\pmod{16},\\
            \!(1,3,1,3) & \text{if } p \equiv 3 \hspace{-2mm}\pmod{16},\\
            \!(3,3,1,2) & \text{if } p \equiv 5 \hspace{-2mm}\pmod{16},\\
            \!(1,3,1,2) & \text{if } p \equiv 7 \hspace{-2mm}\pmod{16},\\
            \!(2,2,1,1) & \text{if } p \equiv 9 \hspace{-2mm}\pmod{16},\\
            \!(1,3,1,1) & \text{if } p \equiv 11 \hspace{-2mm}\pmod{16},\\
            \!(3,3,2,1) & \text{if } p \equiv 13 \hspace{-2mm}\pmod{16},
        \end{array}
    \right.\vspace{-0.1mm}
\]
we see that $(i,j,h,k) \in S$.
Therefore, this curve is not superspecial.
\end{proof}

\subsection{Plane quintic curves with cyclic automorphism groups of order 10}\label{subsec:type6}\vspace{0.5mm}
In this subsection, let us consider a plane quintic curve with an automorphism group containing a subgroup isomorphic to $\mathbb{Z}/10\mathbb{Z}$.
Such a curve can be written as\vspace{-1mm}
\begin{equation}\label{eq:type6}
    C: X^5+Y^5+XZ^4+rX^3Z^2 = 0 \ \text{ with }\ r \neq \pm 2.
\end{equation}
by \cite[Proposition 8]{BB}.
We note that the constraint $r \neq \pm 2$ is imposed so that $C$ is a curve (i.e., non-singular).
In addition, the automorphism group of $C$ can be described as follows:\vspace{-0.8mm}
\begin{itemize}
    \item If $r^2 = 20$, then the automorphism group of $C$ has order $150$.\vspace{-2.6mm}
    \item If $r = 0$, then the automorphism group of $C$ is isomorphic to $\mathbb{Z}/20\mathbb{Z}$.\vspace{-2.6mm}
    \item Otherwise, the automorphism group of $C$ is isomorphic to $\mathbb{Z}/10\mathbb{Z}$.
\end{itemize}
We begin by summarizing some basic properties of such curves:\vspace{-0.9mm}
\begin{Lem}\label{lem:isom_for_type6}
Let $C'$ be another plane quintic curve defined by $X^5+Y^5+XZ^4+r'X^3Z^2 = 0$ with $r' \neq \pm 2$.
Then, the curve $C$ in \eqref{eq:type6} is isomorphic to $C'$ if and only if $r^2 = r'^2$.\vspace{-0.5mm}
\end{Lem}
\begin{proof}
Let\vspace{-0.5mm}
\[
    F(X,Z) \coloneqq X^5+rX^3Z^2+XZ^4 \ \text{ and }\ G(X,Z) \coloneqq X^5+r'X^3Z^2+XZ^4\vspace{0.7mm}
\]
be binary forms of degree $5$ in $K[X,Z]$.
It is clear that $C$ is isomorphic to the curve defined by $Y^5 = F(X,Z)$ and that $C'$ is isomorphic to the curve defined by $Y^5 = G(X,Z)$.
By Proposition \ref{prop:isomorphic}, they are isomorphic to\\ each other if and only if $F(X,Z)$ and $G(X,Z)$ are equivalent under ${\rm GL}_2(K)$.
Rewriting\vspace{-0.8mm}
\begin{alignat*}{3}
    F(X,Z) &= X(X^2-\alpha Z^2)(X^2-\beta Z^2), & \quad & \alpha+\beta =-r,&\ &\alpha\beta = 1,\\[-1.3mm]
    G(X,Z) &= X(X^2-\alpha'\hspace{-0.3mm}Z^2)(X^2-\beta'\hspace{-0.3mm}Z^2), & \quad & \alpha'\hspace{-0.3mm}+\beta'\hspace{-0.3mm} =-r', &\ & \alpha'\hspace{-0.3mm}\beta'\hspace{-0.3mm} = 1,
\end{alignat*}
it follows from \cite[Corollary 2.1.1]{Ishii} that $F(X,Z)$ and $G(X,Z)$ are equivalent under ${\rm GL}_2(K)$ if and only if\vspace{-0.8mm}
\[
    \frac{\alpha}{\beta} + \frac{\beta}{\alpha} = \frac{\alpha'}{\beta'} + \frac{\beta'}{\alpha'}.\vspace{-0.7mm}
\]
By using the relations $\alpha+\beta=1, \alpha\beta=1$, this condition is equivalent to\vspace{-1mm}
\[
    r^2-2 = \frac{(\alpha+\beta)^2-2\alpha\beta}{\alpha\beta} = \frac{(\alpha'+\beta')^2-2\alpha'\beta'}{\alpha'\beta'} = r'^2-2,
\]
which simplifies to $r^2 = r'^2$, as desired.\vspace{-1.3mm}
\end{proof}
\begin{Cor}\label{cor:rational-type6}
If the plane quintic curve $C$ in \eqref{eq:type6} is superspecial, then $r^2$ belongs to $\mathbb{F}_{p^2}$.\vspace{-1.2mm}
\end{Cor}
\begin{proof}
Suppose that $C$ is superspecial.
As is well known (cf. \cite[Theorem 1.1]{Ekedahl}), any superspecial curve admits a model defined over $\mathbb{F}_{p^2}$.
Hence $C$ is isomorphic to its $p^2$-Frobenius twist\vspace{-1.3mm}
\[
    C^{(p^2)}\hspace{-0.2mm}: X^5+Y^5+r^{p^2}\hspace{-0.5mm}X^3Z^2 +XZ^4 = 0.\vspace{0.4mm}
\]
By Lemma~\ref{lem:isom_for_type6}, this implies that $r^2 = (r^{p^2})^2$.
It follows that $(r^2)^{p^2} = r^2$, and thus $r^2 \in \mathbb{F}_{p^2}$.
\end{proof}

\newpage
Now, fix a root $t \neq 0$ of the quartic equation $t^4+rt^2+1 = 0$.
Then, the plane quintic curve $C$ in \eqref{eq:type6} is isomorphic to the superelliptic curve\vspace{-1.3mm}
\begin{equation}\label{eq:Clambda_for_type6}
    C_\lambda\hspace{-0.2mm}: y^5 = x(x^2-1)(x^2-\lambda) \ \text{ with }\,\lambda \coloneqq t^4\vspace{0.4mm}
\end{equation}
via the map $(X:Y:Z) = (x:-y:t)$.
Since $r = -(t^2+t^{-2})$, we have that $r^2 = t^4+2+t^{-4} = \lambda+2+1/\lambda$, which yields the important relation\vspace{-0.2mm}
\begin{equation}\label{eq:lambda_for_type6}
    \lambda+\frac{1}{\lambda} = r^2\hspace{-0.3mm}-2.
\end{equation}
Note that $\lambda \neq 0$ since $t \neq 0$, and that $\lambda \neq 1$ follows from the assumption $r \neq \pm 2$.
The following proposition characterizes the superspeciality of $C_\lambda$ in terms of Gaussian hypergeometric series:\vspace{-1.2mm}
\begin{Prop}\label{prop:gauss_for_type6}
For the superelliptic curve $C_\lambda$ in \eqref{eq:Clambda_for_type6}, the following statements hold:\vspace{-1.6mm}
\begin{enumerate}
\item[(1)] If $p \not\equiv 4 \pmod{5}$, then there exists no $\lambda \neq 1$ such that $C_\lambda$ is superspecial.\vspace{-2.8mm}
\item[(2)] If $p \equiv 4 \pmod{5}$, then $C_\lambda$ is superspecial if and only if\vspace{-1.2mm}
\begin{equation}\label{eq:Glambda}
    G(\lambda) \coloneqq G^{((3p-7)/10)}\scalebox{0.95}{$\displaystyle\left(\frac{3}{5},\frac{7}{10},\frac{11}{10}\ \middle|\ \lambda\right)$} = 0,\vspace{-0.7mm}
\end{equation}
where $G^{(d)}(a,b,c \mid z)$ is defined in \eqref{eq:Gaussian}.\vspace{-0.6mm}
\end{enumerate}
\end{Prop}
\begin{proof}
We apply \hspace{-0.1mm}Theorem~\ref{thm:Clambda} for the case $(n,r) = (5,2)$.
In the following, we divide the proof according to the value of $p \bmod{5}$.\vspace{-1.1mm}
\begin{itemize}
\item If $p \equiv 1 \pmod{5}$, then we see that $(i,j,h,k) \coloneqq (3,3,1,1)$ lies in $T$, where $T$ is the set in \eqref{eq:T}.
Hence, it follows that $C_\lambda$ is superspecial only if\vspace{-1.5mm}
\[
    G^{((p-1)/5)}\scalebox{0.95}{$\displaystyle\left(\frac{3}{5},\frac{1}{5},\frac{3}{5}\ \middle|\ \lambda\right)$} = 0.\vspace{-0.6mm}
\]
Since this left-hand side is equal to $(1-\lambda)^{(p-1)/5}$ by the binomial theorem, there does not exist $\lambda \neq 1$ such that $C_\lambda$ is superspecial.\vspace{-1.5mm}
\item If $p \equiv 2 \pmod{5}$, then we see that $(i,j,h,k) \coloneqq (4,3,2,1)$ lies in $T$, where $T$ is the set in \eqref{eq:T}.
Hence, it follows that $C_\lambda$ is superspecial only if\vspace{-1.5mm}
\[
    G^{((3p-1)/5)}\scalebox{0.95}{$\displaystyle\left(\frac{3}{5},\frac{1}{5},\frac{3}{5}\ \middle|\ \lambda\right)$} = 0.\vspace{-0.6mm}
\]
Since this left-hand side is equal to $(1-\lambda)^{(3p-1)/5}$ by the binomial theorem, there does not exist $\lambda \neq 1$ such that $C_\lambda$ is superspecial.\vspace{-1.5mm}
\item If $p \equiv 3 \pmod{5}$, then we see that $(i,j,h,k) \coloneqq (3,4,1,2)$ lies in $T$, where $T$ is the set in \eqref{eq:T}.
Hence, it follows that $C_\lambda$ is superspecial only if\vspace{-1.5mm}
\[
    G^{((p-3)/5)}\scalebox{0.95}{$\displaystyle\left(\frac{4}{5},\frac{3}{5},\frac{4}{5}\ \middle|\ \lambda\right)$} = 0.\vspace{-0.6mm}
\]
Since this left-hand side is equal to $(1-\lambda)^{(p-3)/5}$ by the binomial theorem, there does not exist $\lambda \neq 1$ such that $C_\lambda$ is superspecial.\vspace{-1.5mm}
\item If $p \equiv 4 \pmod{5}$, then we see that the set $T$ in \eqref{eq:T} is given by $T = \{(2,3,1,2),(3,2,2,1)\}$.
Therefore, it follows that $C_\lambda$ is superspecial if and only if\vspace{-1mm}
\[
    G^{((3p-7)/10)}\scalebox{0.95}{$\displaystyle\left(\frac{3}{5},\frac{7}{10},\frac{11}{10}\ \middle|\ \lambda\right)$} = G^{((p-1)/2)}\scalebox{0.95}{$\displaystyle\left(\frac{2}{5},\frac{1}{2},\frac{11}{10}\ \middle|\ \lambda\right)$} = 0.\vspace{-0.8mm}
\]
On the other hand, a straightforward computation with Euler’s transformation formula yields\vspace{-0.6mm}
\[
    G^{((p-1)/2)}\scalebox{0.95}{$\displaystyle\left(\frac{2}{5},\frac{1}{2},\frac{11}{10}\ \middle|\ \lambda\right)$} = (1-\lambda)^{(p+1)/5}G^{((3p-7)/10)}\scalebox{0.95}{$\displaystyle\left(\frac{3}{5},\frac{7}{10},\frac{11}{10}\ \middle|\ \lambda\right)$},\vspace{-0.9mm}
\]
and hence, the two conditions above are equivalent since $\lambda \neq 1$.
This tells us that $C_\lambda$ is superspecial if and only if $G(\lambda) = 0$.\vspace{-0.6mm}
\end{itemize}
Combining these four cases, we obtain the desired assertion.\vspace{-2mm}
\end{proof}

\newpage
The following lemma collects some properties of the polynomial $G(\lambda)$ defined in \eqref{eq:Glambda} by using arguments similar to those in \cite[Section 1.4]{IKO}.\vspace{-1.3mm}
\begin{Lem}\label{lem:simple_for_type6}
If $p \equiv 4 \pmod{5}$, then the following statements hold:\vspace{-1.1mm}
\begin{enumerate}
\item[(1)] The degree of $G(\lambda)$ is equal to $(3p-7)/10$.\vspace{-2.6mm}
\item[(2)] The roots of $G(\lambda)$ are different from $0$ and $1$.\vspace{-2.6mm}
\item[(3)] The polynomial $G(\lambda)$ does not have any multiple root.\vspace{-0.3mm}
\end{enumerate}
\end{Lem}
\begin{proof}
(1) By definition, the degree of $G(\lambda)$ is at most $(3p-7)/10$.
The $\lambda^{(3p-7)/10}$-coefficient of $G(\lambda)$ is\vspace{-0.9mm}
\[
    \frac{\bigl(\frac{3}{5}\,;\hspace{-0.2mm}\frac{3p-7}{10}\bigr)\hspace{-0.3mm}\bigl(\frac{7}{10}\,;\hspace{-0.2mm}\frac{3p-7}{10}\bigr)}{\bigl(\frac{11}{10}\,;\hspace{-0.2mm}\frac{3p-7}{10}\bigr)\hspace{-0.3mm}\bigl(1\,;\hspace{-0.2mm}\frac{3p-7}{10}\bigr)} \not\equiv 0 \pmod{p},
\]
and hence $G(\lambda)$ is of degree $(3p-7)/10$.\par
\hspace{5.8mm}(2) It is clear that the constant term of $G(\lambda)$ is equal to $1$, and hence $G(0) = 1$.
In addition, it follows from the Chu-Vandermonde identity that\vspace{-1.9mm}
\[
    G(1) \equiv G^{((3p-7)/10)}\scalebox{0.95}{$\displaystyle\left(\frac{3}{5},-\frac{3p-7}{10},\frac{11}{10}\ \middle|\ 1\right)$} = \frac{\bigl(\frac{1}{2}\,;\hspace{-0.2mm}\frac{3p-7}{10}\bigr)}{\bigl(\frac{11}{10}\,;\hspace{-0.2mm}\frac{3p-7}{10}\bigr)} \not\equiv 0 \pmod{p}.\vspace{-1.9mm}
\]
This means that $G(\lambda)$ does not have roots $\lambda=0,1$.\par
\hspace{5.8mm}(3) The Euler's hypergeometric differential equation implies that $\mathcal{D}G(\lambda)=0$ where\vspace{-1.4mm}
\[
    \mathcal{D} \coloneqq 50z(1-z)\frac{d^2}{dz^2} + 5(11-23z)\frac{d}{dz}-21.\vspace{-0.3mm}
\]
Assume that $G(\lambda)$ has a multiple root at $\lambda = \lambda_0$.
Since $G(\lambda_0)=G'(\lambda_0)=0$, the differential equation above yields $G''(\lambda_0) = 0$.
Repeating this argument inductively, we see that all derivatives of $G(\lambda)$ vanish at $\lambda = \lambda_0$.
This contradicts the fact that $G(\lambda)$ is a non-zero polynomial, and thus all roots of $G(\lambda)$ are simple.
\end{proof}

We now prove the following main theorem (Theorem \ref{thm:main-10}), which gives the number of isomorphism classes of superspecial plane quintic curves with a cyclic automorphism group of order $10$.\vspace{-1.3mm}
\begin{proof}[The proof of Theorem \ref{thm:main-10}]
We recall that every plane quintic curve with an automorphism group containing a subgroup isomorphic to $\mathbb{Z}/10\mathbb{Z}$ is isomorphic to $C_\lambda$ in \eqref{eq:Clambda_for_type6} for some $\lambda \neq 0,1$.
Hence, this theorem follows in the case $p \not\equiv 4 \pmod{5}$ from Proposition~\ref{prop:gauss_for_type6}(1).
In the case $p \equiv 4 \pmod{5}$, it follows from Proposition~\ref{prop:gauss_for_type6}(2) and Lemma~\ref{lem:simple_for_type6} that there are exactly $(3p-7)/10$ values of $\lambda \neq 0,1$ for which $C_\lambda$ is superspecial.
Moreover, the number of those for which $\hspace{-0.2mm}{\rm Aut}(C_\lambda) \cong \mathbb{Z}/10\mathbb{Z}$ is given by\vspace{-0.7mm}
\begin{equation}\label{eq:number_of_type6}
    \left\{
        \begin{array}{ll}
            (3p-7)/10-2 & \text{ if } p\equiv 9 \!\!\pmod{20},\\[-0.2mm]
            (3p-7)/10-3 & \text{ if } p\equiv 19 \!\!\pmod{20}.
        \end{array}
    \right.\vspace{-0.8mm}
\end{equation}
Indeed, by the discussion at the beginning of Section~\ref{subsec:type6} and the relation \eqref{eq:lambda_for_type6},\vspace{-1.1mm}
\begin{itemize}
    \item If $\lambda = 9 \pm 4\sqrt{5}$, then the automorphism group of $C_\lambda$ has order 150.
    In this case $C_\lambda\hspace{-0.1mm}$ corresponds to the plane quintic of Type \textit{1}, which is superspecial if and only if $p \equiv 4 \pmod{5}$ by Proposition \ref{prop:type1-2}(1).\vspace{-1.6mm}
    \item If $\lambda = -1$, then the automorphism group of $C_\lambda\hspace{-0.1mm}$ is isomorphic to $\mathbb{Z}/20\mathbb{Z}$.
    In this case $C_\lambda\hspace{-0.1mm}$ corresponds to the plane quintic of Type \textit{4}, which is superspecial if and only if $p \equiv 19 \pmod{20}$ by Proposition \ref{prop:type3-5}(2).\vspace{-1.6mm}
    \item Otherwise, the automorphism group of $C_\lambda$ is isomorphic to $\mathbb{Z}/10\mathbb{Z}$.\vspace{-0.8mm}
\end{itemize}
By Lemma~\ref{lem:isom_for_type6} and the relation~\eqref{eq:lambda_for_type6}, we have $C_\lambda \cong C_{\lambda'}$\hspace{-0.1mm} if and only if $\lambda' \hspace{-0.1mm}\in\hspace{-0.1mm} \{\lambda,1/\lambda\}$.
Therefore, the number of isomorphism classes of superspecial curves $C_\lambda$ with ${\rm Aut}(C_\lambda) \cong \mathbb{Z}/10\mathbb{Z}$ is half of \eqref{eq:number_of_type6}, as desired.\vspace{-1.1mm}
\end{proof}
\begin{Rmk}
One can also show that superspecial plane quintic curves whose automorphism groups contain a subgroup isomorphic to $\mathbb{Z}/10\mathbb{Z}$ exist only if $p \equiv 4 \pmod{5}$, as follows.
For a plane quintic curve $C$ in \eqref{eq:type6}, we choose $\alpha,\beta \in K$ satisfying $\alpha^5 = -4/(r^2-4)$ and $\beta^2 = r^2-4$.
Then, there exists a morphism\vspace{-0.9mm}
\[
    C \longrightarrow H, \quad (X:Y:Z) \longmapsto \biggl(\frac{\alpha Y}{X},\frac{rX^2+2Z^2}{\beta X^2}\biggr) \eqqcolon (u,v),\vspace{-0.3mm}
\]
where $H$ is the genus-2 curve defined by $v^2 = u^5+1$.
By \cite[Corollary 2.8]{KNT}, if $C$ is superspecial, then $H$ must also be superspecial.
Since $H$ is superspecial if and only if $p \equiv 4 \pmod{5}$, we obtain the desired assertion.
\end{Rmk}

\newpage
\subsection{Plane quintic curves with cyclic automorphism groups of order 8}\label{subsec:type8}
In this subsection, let us consider a plane quintic curve with an automorphism group containing a subgroup isomorphic to $\mathbb{Z}/8\mathbb{Z}$.
Such a curve can be written as\vspace{-1.2mm}
\begin{equation}\label{eq:type8}
    C: X^5+Y^4Z+XZ^4+rX^3Z^2 = 0 \ \text{ with }\ r \neq \pm 2.
\end{equation}
by \cite[Corollary 7]{BB}.
We note that the constraint $r \neq \pm 2$ is imposed so that $C$ is a curve (i.e., non-singular).
In addition, the automorphism group of $C$ can be described as follows:\vspace{-0.8mm}
\begin{itemize}
    \item If $r = 0$, then the automorphism group of $C$ is isomorphic to $\mathbb{Z}/16\mathbb{Z}$.\vspace{-2.6mm}
    \item Otherwise, the automorphism group of $C$ is isomorphic to $\mathbb{Z}/8\mathbb{Z}$.
\end{itemize}
A necessary and sufficient condition for such curves to be isomorphic is given as follows:\vspace{-1.2mm}
\begin{Lem}\label{lem:isom_for_type8}
Let $C'$ be another plane quintic curve defined by $\hspace{-0.2mm}X^5+Y^4Z+XZ^4+r'\hspace{-0.2mm}X^3Z^2 = 0$ with $r' \neq \pm 2$.
Then, the curve $C$ in \eqref{eq:type8} is isomorphic to $C'$ if and only if $r^2 = r'^2$.\vspace{-1mm}
\end{Lem}
\begin{proof}
We see that $C$ (resp. $C'$) is isomorphic to the curve defined by $Y^4Z = F(X,Z)$ (resp. $Y^4Z = G(X,Z)$) where\vspace{-1mm}
\[
    F(X,Z) \coloneqq X^5+rX^3Z^2+XZ^4 \ \text{ and }\ G(X,Z) \coloneqq X^5+r'X^3Z^2+XZ^4\vspace{0.9mm}
\]
are binary forms of degree $5$ in $K[X,Z]$.
Hence, it follows from Proposition \ref{prop:isomorphic} that $C$ and $C'$ are isomorphic to each other if and only if $F(X,Z)$ and $G(X,Z)$ are equivalent under ${\rm GL}_2(K)$.
The remainder of the proof is exactly the same as that of Lemma \ref{lem:isom_for_type6}.
\end{proof}

By the above lemma, one can show, similarly to Corollary \ref{cor:rational-type6}, that if the curve $C$ in \eqref{eq:type8} is superspecial, then $r^2 \in \mathbb{F}_{p^2}$.
On the other hand, we can prove the following stronger statement:\vspace{-1.4mm}
\begin{Prop}\label{prop:rational_for_type8}
If the plane quintic curve $C$ in \eqref{eq:type8} is superspecial, then $r+2$ and $r-2$ are both squares in $\mathbb{F}_{p^2}$\hspace{-0.2mm}. In particular, $r$ belongs to $\mathbb{F}_{p^2}$.\vspace{-1mm}
\end{Prop}
\begin{proof}
We fix a root $t \neq 0$ of the equation $t^{16}+rt^8+1 = 0$.
Then, the curve $C$ in \eqref{eq:type8} is isomorphic to the superelliptic curve\vspace{-1.1mm}
\begin{equation}\label{eq:Clambda_for_type8}
    C_\lambda: y^4 = x(x^2-1)(x^2-\lambda) \ \text{ with }\, \lambda \coloneqq t^{16}\vspace{1.2mm}
\end{equation}
via the map $(X:Y:Z) = (-tx:y:t^5)$.
This curve admits the involution $(x,y) \mapsto\hspace{-0.3mm} (x,-y)$, whose quotient curve is the genus-2 curve\vspace{-1mm}
\begin{equation}\label{eq:H}
    H: v^2 = u(u^2-1)(u^2-\lambda)\vspace{1mm}
\end{equation}
where $u \coloneqq x,\,v \coloneqq y^2$.
It follows from \cite[Corollary 2.8]{KNT} that, if $C$ is superspecial, then $H$ is also superspecial.
By \cite[Main \hspace{-0.1mm}Theorem \hspace{-0.2mm}A(1)]{Ohashi}, the superspeciality of $H$ implies that $\hspace{-0.2mm}\sqrt{\lambda} = t^8$ is a square in $\mathbb{F}_{p^2}\hspace{-0.2mm}$, which means that $t^4$ belongs to $\mathbb{F}_{p^2}$.
Therefore, the values\vspace{-1.3mm}
\[
    r \pm 2 = -(t^8 \mp 2 + t^{-8}) = -(t^4 \mp t^{-4})^2
\]
are squares in $\mathbb{F}_{p^2}$, since $-1$ is a square in $\mathbb{F}_{p^2}$.
\end{proof}

In the following, instead of working with the plane quintic curve $C$ in \eqref{eq:type8}, we consider the superelliptic curve $C_\lambda$, which was constructed in the proof of Proposition \ref{prop:rational_for_type8}.
Since $r = -(t^8+t^{-8})$ and $\lambda = t^{16}$, we have the important relation\vspace{-0.8mm}
\begin{equation}\label{eq:lambda_for_type8}
    \lambda+\frac{1}{\lambda} = r^2\hspace{-0.3mm}-2.
\end{equation}
Note that $\lambda \neq 0$ since $t \neq 0$, and that $\lambda \neq 1$ follows from the assumption $r \neq \pm 2$.
The following proposition characterizes the superspeciality of $C_\lambda$ in terms of Gaussian hypergeometric series:\vspace{-1.3mm}
\begin{Prop}\label{prop:gauss_for_type8}
For the superelliptic curve $C_\lambda$ in \eqref{eq:Clambda_for_type8}, the following statements hold:\vspace{-1.5mm}
\begin{enumerate}
\item[(1)] If $p \equiv 1 \pmod{8}$, then $C_\lambda$ is superspecial if and only if\vspace{-1.3mm}
\begin{align*}
    G^{((p-1)/4)}\scalebox{0.95}{$\displaystyle\left(\frac{1}{2},\frac{1}{4},\frac{3}{4}\ \middle|\ \lambda\right)$} &= G^{((p-1)/8)}\scalebox{0.95}{$\displaystyle\left(\frac{1}{4},\frac{1}{8},\frac{7}{8}\ \middle|\ \lambda\right)$} = G^{((p-1)/8)}\scalebox{0.95}{$\displaystyle\left(\frac{3}{4},\frac{1}{8},\frac{3}{8}\ \middle|\ \lambda\right)$}\\[-0.2mm]
    = G^{((3p-3)/8)}\scalebox{0.95}{$\displaystyle\left(\frac{3}{4},\frac{3}{8},\frac{5}{8}\ \middle|\ \lambda\right)$} &= G^{((p-9)/8)}\scalebox{0.95}{$\displaystyle\left(\frac{3}{4},\frac{9}{8},\frac{11}{8}\ \middle|\ \lambda\right)$} = G^{((3p-11)/8)}\scalebox{0.95}{$\displaystyle\left(\frac{3}{4},\frac{11}{8},\frac{13}{8}\ \middle|\ \lambda\right)$} = 0.
\end{align*}
\item[(2)] If $p \equiv 3 \pmod{8}$, then $C_\lambda$ is superspecial if and only if\vspace{-1mm}
\begin{align*}
    G^{((p-3)/4)}\scalebox{0.95}{$\displaystyle\left(\frac{1}{2},\frac{3}{4},\frac{5}{4}\ \middle|\ \lambda\right)$} &= G^{((3p-1)/8)}\scalebox{0.95}{$\displaystyle\left(\frac{1}{4},\frac{1}{8},\frac{7}{8}\ \middle|\ \lambda\right)$} = G^{((p-3)/8)}\scalebox{0.95}{$\displaystyle\left(\frac{1}{4},\frac{3}{8},\frac{9}{8}\ \middle|\ \lambda\right)$}\\[-0.2mm]
    = G^{((p-3)/8)}\scalebox{0.95}{$\displaystyle\left(\frac{3}{4},\frac{3}{8},\frac{5}{8}\ \middle|\ \lambda\right)$} &= G^{((p-11)/8)}\scalebox{0.95}{$\displaystyle\left(\frac{3}{4},\frac{11}{8},\frac{13}{8}\ \middle|\ \lambda\right)$} = 0.\\[-7.5mm]
\end{align*}
\item[(3)] If $p \equiv 5 \pmod{8}$, then $C_\lambda$ is superspecial if and only if\vspace{-1mm}
\begin{align*}
    G^{((p-1)/4)}\scalebox{0.95}{$\displaystyle\left(\frac{1}{2},\frac{1}{4},\frac{3}{4}\ \middle|\ \lambda\right)$} &= G^{((5p-1)/8)}\scalebox{0.95}{$\displaystyle\left(\frac{3}{4},\frac{1}{8},\frac{3}{8}\ \middle|\ \lambda\right)$} = G^{((p-5)/8)}\scalebox{0.95}{$\displaystyle\left(\frac{3}{4},\frac{5}{8},\frac{7}{8}\ \middle|\ \lambda\right)$}\\[-0.2mm]
    = G^{((3p-7)/8)}\scalebox{0.95}{$\displaystyle\left(\frac{3}{4},\frac{7}{8},\frac{9}{8}\ \middle|\ \lambda\right)$} &= G^{((5p-9)/8)}\scalebox{0.95}{$\displaystyle\left(\frac{3}{4},\frac{9}{8},\frac{11}{8}\ \middle|\ \lambda\right)$} = 0.\\[-7.5mm]
\end{align*}
\item[(4)] If $p \equiv 7 \pmod{8}$, then $C_\lambda$ is superspecial if and only if\vspace{-1mm}
\[
    G^{((p-3)/4)}\scalebox{0.95}{$\displaystyle\left(\frac{1}{2},\frac{3}{4},\frac{5}{4}\ \middle|\ \lambda\right)$} = G^{((p-7)/8)}\scalebox{0.95}{$\displaystyle\left(\frac{3}{4},\frac{7}{8},\frac{9}{8}\ \middle|\ \lambda\right)$} = 0.\vspace{-2.3mm}
\]
\end{enumerate}
\end{Prop}
\begin{proof}
(1) If $p \equiv 1 \pmod{8}$, then we see that the set $T$ in \eqref{eq:T} is given by\vspace{-1mm}
\[
    T = \{(1,1,1,1),(2,2,1,1),(2,2,2,2),(3,3,1,1),(3,3,1,3),(3,3,2,2),(3,3,3,1),(3,3,3,3)\}.
\]
Therefore, it follows from Theorem~\ref{thm:Clambda} that $C_\lambda$ is superspecial if and only if\vspace{-1.1mm}
\begin{align*}
    G^{((p-1)/8)}\scalebox{0.95}{$\displaystyle\left(\frac{1}{4},\frac{1}{8},\frac{7}{8}\ \middle|\ \lambda\right)$} &= G^{((p-1)/4)}\scalebox{0.95}{$\displaystyle\left(\frac{1}{2},\frac{1}{4},\frac{3}{4}\ \middle|\ \lambda\right)$}\\[-0.2mm]
    = G^{((p-1)/8)}\scalebox{0.95}{$\displaystyle\left(\frac{3}{4},\frac{1}{8},\frac{3}{8}\ \middle|\ \lambda\right)$} &= G^{((p-9)/8)}\scalebox{0.95}{$\displaystyle\left(\frac{3}{4},\frac{9}{8},\frac{11}{8}\ \middle|\ \lambda\right)$} = G^{((5p-5)/8)}\scalebox{0.95}{$\displaystyle\left(\frac{3}{4},\frac{5}{8},\frac{7}{8}\ \middle|\ \lambda\right)$}\\[-0.2mm]
    = G^{((3p-11)/8)}\scalebox{0.95}{$\displaystyle\left(\frac{3}{4},\frac{11}{8},\frac{13}{8}\ \middle|\ \lambda\right)$} &= G^{((3p-3)/8)}\scalebox{0.95}{$\displaystyle\left(\frac{3}{4},\frac{3}{8},\frac{5}{8}\ \middle|\ \lambda\right)$} = 0.\\[-6.5mm]
\end{align*}
Note that both $(2,2,1,1)$ and $(2,2,2,2)$ in $T$ determine the second polynomial of the above equalities.
Also, a straightforward computation with Euler’s transformation formula yields\vspace{-1.3mm}
\[
    G^{((5p-5)/8)}\scalebox{0.95}{$\displaystyle\left(\frac{3}{4},\frac{5}{8},\frac{7}{8}\ \middle|\ \lambda\right)$} = (1-\lambda)^{(p-1)/2}G^{((p-1)/8)}\scalebox{0.95}{$\displaystyle\left(\frac{1}{4},\frac{1}{8},\frac{7}{8}\ \middle|\ \lambda\right)$},\vspace{-1.1mm}
\]
and hence, the first and fifth polynomial of the above equalities have the same roots since $\lambda \neq 1$.\par
\hspace{5.8mm}(2) If $p \equiv 3 \pmod{8}$, then we see that the set $T$ in \eqref{eq:T} is given by\vspace{-1.2mm}
\[
    T = \{(1,3,1,1),(1,3,1,3),(2,2,1,2),(2,2,2,1),(3,1,1,1),(3,1,3,1)\}.\vspace{-0.1mm}
\]
Therefore, it follows from Theorem~\ref{thm:Clambda} that $C_\lambda$ is superspecial if and only if\vspace{-1.1mm}
\begin{align*}
    G^{((p-11)/8)}\scalebox{0.95}{$\displaystyle\left(\frac{3}{4},\frac{11}{8},\frac{13}{8}\ \middle|\ \lambda\right)$} &= G^{((p-3)/8)}\scalebox{0.95}{$\displaystyle\left(\frac{3}{4},\frac{3}{8},\frac{5}{8}\ \middle|\ \lambda\right)$} = G^{((p-3)/4)}\scalebox{0.95}{$\displaystyle\left(\frac{1}{2},\frac{3}{4},\frac{5}{4}\ \middle|\ \lambda\right)$}\\[-0.2mm]
    = G^{((p-3)/8)}\scalebox{0.95}{$\displaystyle\left(\frac{1}{4},\frac{3}{8},\frac{9}{8}\ \middle|\ \lambda\right)$} &= G^{((3p-1)/8)}\scalebox{0.95}{$\displaystyle\left(\frac{1}{4},\frac{1}{8},\frac{7}{8}\ \middle|\ \lambda\right)$} = 0.\\[-6.4mm]
\end{align*}
Note that both $(2,2,1,2)$ and $(2,2,2,1)$ in $T$ determine the third polynomial of the above equalities.\par
\hspace{5.8mm}(3) If $p \equiv 5 \pmod{8}$, then we see that the set $T$ in \eqref{eq:T} is given by\vspace{-1.3mm}
\[
    T = \{(2,2,1,1),(2,2,2,2),(3,3,1,2),(3,3,2,1),(3,3,2,3),(3,3,3,2)\}.\vspace{-0.2mm}
\]
Therefore, it follows from Theorem~\ref{thm:Clambda} that $C_\lambda$ is superspecial if and only if\vspace{-1.1mm}
\begin{align*}
    G^{((p-1)/4)}\scalebox{0.95}{$\displaystyle\left(\frac{1}{2},\frac{1}{4},\frac{3}{4}\ \middle|\ \lambda\right)$} &= G^{((p-5)/8)}\scalebox{0.95}{$\displaystyle\left(\frac{3}{4},\frac{5}{8},\frac{7}{8}\ \middle|\ \lambda\right)$} = G^{((5p-1)/8)}\scalebox{0.95}{$\displaystyle\left(\frac{3}{4},\frac{1}{8},\frac{3}{8}\ \middle|\ \lambda\right)$}\\[-0.2mm]
    = G^{((5p-9)/8)}\scalebox{0.95}{$\displaystyle\left(\frac{3}{4},\frac{9}{8},\frac{11}{8}\ \middle|\ \lambda\right)$} &= G^{((3p-7)/8)}\scalebox{0.95}{$\displaystyle\left(\frac{3}{4},\frac{7}{8},\frac{9}{8}\ \middle|\ \lambda\right)$} = 0.\\[-6.4mm]
\end{align*}
Note that both $(2,2,1,1)$ and $(2,2,2,2)$ in $T$ determine the first polynomial of the above equalities.\newpage\par
\hspace{5.8mm}(4) If $p \equiv 7 \pmod{8}$, then we see that the set $T$ in \eqref{eq:T} is given by\vspace{-1.3mm}
\[
    T = \{(1,3,1,2),(2,2,1,2),(2,2,2,1),(3,1,2,1)\}.\vspace{-0.1mm}
\]
Therefore, it follows from Theorem~\ref{thm:Clambda} that $C_\lambda$ is superspecial if and only if\vspace{-1.4mm}
\[
    G^{((p-7)/8)}\scalebox{0.95}{$\displaystyle\left(\frac{3}{4},\frac{7}{8},\frac{9}{8}\ \middle|\ \lambda\right)$} = G^{((p-3)/4)}\scalebox{0.95}{$\displaystyle\left(\frac{1}{2},\frac{3}{4},\frac{5}{4}\ \middle|\ \lambda\right)$} = G^{((5p-3)/8)}\scalebox{0.95}{$\displaystyle\left(\frac{1}{4},\frac{3}{8},\frac{9}{8}\ \middle|\ \lambda\right)$}.\vspace{-1mm}
\]
Note that both $(2,2,1,2)$ and $(2,2,2,1)$ in $T$ determine the second polynomial of the above equalities.
Also, a straightforward computation with Euler's transformation formula yields\vspace{-1.3mm}
\[
    G^{((5p-3)/8)}\scalebox{0.95}{$\displaystyle\left(\frac{1}{4},\frac{3}{8},\frac{9}{8}\ \middle|\ \lambda\right)$} = (1-\lambda)^{(p+1)/2}G^{((p-7)/8)}\scalebox{0.95}{$\displaystyle\left(\frac{3}{4},\frac{7}{8},\frac{9}{8}\ \middle|\ \lambda\right)$},\vspace{-1.1mm}
\]
and hence, the first and third polynomial of the above equalities have the same roots since $\lambda \neq 1$.\vspace{-1.8mm}
\end{proof}
\begin{Rmk}
Applying Proposition \ref{prop:gauss_for_type8} for the case $(n,r) = (2,2)$, we see that the hyperelliptic curve $H$ in \eqref{eq:H} is superspecial if and only if $h(\lambda)=0$, where\vspace{-1mm}
\begin{equation}\label{eq:h}
    h(\lambda) \coloneqq \left\{
        \begin{array}{ll}
            G^{((p-1)/4)}(1/2,1/4,3/4 \mid \lambda) & \text{if }\, p \equiv 1 \!\!\pmod{4},\\
            G^{((p-3)/4)}(1/2,3/4,5/4 \mid \lambda) & \text{if }\, p \equiv 3 \!\!\pmod{4}.
        \end{array}
    \right.\vspace{-0.7mm}
\end{equation}
This coincides with the condition that the first polynomial vanishes in Proposition \ref{prop:gauss_for_type8} for each case.
\end{Rmk}

Let $G(\lambda)$ be the gcd of polynomials in Proposition \ref{prop:gauss_for_type8} for each case. For example, we define\vspace{-0.8mm}
\[
    G(\lambda) \coloneqq {\rm gcd}\biggl(G^{((p-3)/4)}\scalebox{0.95}{$\displaystyle\left(\frac{1}{2},\frac{3}{4},\frac{5}{4}\ \middle|\ \lambda\right)$},\,G^{((p-7)/8)}\scalebox{0.95}{$\displaystyle\left(\frac{3}{4},\frac{7}{8},\frac{9}{8}\ \middle|\ \lambda\right)$}\biggr)\vspace{-0.2mm}
\]
if $p \equiv 7 \pmod{8}$.
This polynomial is separable and has no roots equal to $0$ or $1$, since $h(\lambda)$ defined in \eqref{eq:h} is separable and has no roots equal to $0$ or $1$ by \cite[Proposition 1.14]{IKO}.
Hence, we have the following theorem as an analogy of Theorem \ref{thm:main-10}.\vspace{-1.4mm}
\begin{Thm}\label{prethm:main-8}
Let $G(\lambda)$ be the gcd of polynomials in Proposition \ref{prop:gauss_for_type8} for each case.
Then, the number of isomorphism classes of superspecial plane quintic curves with cyclic automorphism group of order $8$ equals\vspace{-0.9mm}
\[
    \left\{ 
        \begin{array}{ll}
            \hspace{-0.3mm}\frac{1}{2}(\deg{G(\lambda)}-1) & \text{ if } p \equiv 15 \!\!\pmod{16},\\[0.5mm]
            \hspace{-0.3mm}\frac{1}{2}\hspace{-0.3mm}\deg{G(\lambda)} & \text{ otherwise}
        \end{array}
    \right.\vspace{-1mm}
\]
in characteristic $p > 13$.\vspace{-1.1mm}
\end{Thm}
\begin{proof}
Recall that every plane quintic curve with an automorphism group containing a subgroup isomorphic to $\mathbb{Z}/8\mathbb{Z}$ is isomorphic to $C_\lambda$ in \eqref{eq:Clambda_for_type8} for some $\lambda \neq 0,1$.
By the definition of $G(\lambda)$, the curve $C_\lambda$ is superspecial if and only if $G(\lambda)=0$.
Since $G(\lambda)$ is separable and has no roots equal to $0$ or $1$, the number of $\lambda \neq 0,1$ for which $C_\lambda$ is superspecial and ${\rm Aut}(C_\lambda) \cong \mathbb{Z}/8\mathbb{Z}$ is given by\vspace{-1mm}
\begin{equation}\label{eq:number_of_type8}
    \left\{
        \begin{array}{ll}
            \hspace{-0.3mm}\deg{G(\lambda)}-1 & \text{ if } p\equiv 15 \!\!\pmod{16},\\[-0.2mm]
            \hspace{-0.3mm}\deg{G(\lambda)} & \text{ otherwise}.
        \end{array}
    \right.\vspace{-1mm}
\end{equation}
Indeed, by the discussion at the beginning of Section \ref{subsec:type8} and the relation \eqref{eq:Clambda_for_type8},\vspace{-1.2mm}
\begin{itemize}
    \item If $\lambda = -1$, then the automorphism group of $C_\lambda$ is isomorphic to $\mathbb{Z}/16\mathbb{Z}$.
    In this case $C_\lambda$ corresponds to the plane quintic of Type \textit{5}, which is superspecial if and only if $p \equiv 15 \pmod{16}$ by Proposition \ref{prop:type3-5}(3).\vspace{-1.6mm}
    \item Otherwise, the automorphism group of $C_\lambda$ is isomorphic to $\mathbb{Z}/8\mathbb{Z}$.\vspace{-0.6mm}
\end{itemize}
By Lemma \ref{lem:isom_for_type8} and the relation \eqref{eq:Clambda_for_type8}, we have $C_\lambda \cong C_{\lambda'}$\hspace{-0.1mm} if and only if $\lambda' \in \{\lambda,1/\lambda\}$.
Therefore, the number of isomorphism classes of superspecial curves $C_\lambda$ with ${\rm Aut}(C_\lambda) \cong \mathbb{Z}/8\mathbb{Z}$ is half of \eqref{eq:number_of_type8}, as desired.
\end{proof}

Thanks to the above theorem, we can determine the number of isomorphism classes of superspecial plane quintic curves with cyclic automorphism group of order $8$ by $G(\lambda)$, which can be computed in $\widetilde{O}(p)$ operations over $\mathbb{F}_p\hspace{-0.1mm}$ by an argument similar to the proof of \cite[Theorem 3.2.9]{OK}.
Finally, we implemented our algorithm in Magma for explicitly computing the number.
We ran it on a machine equipped with an AMD EPYC 7742 CPU and 2TB of RAM in every characteristic $13 < p < 10000$.
The experimental results for $p \equiv 7 \pmod{8}$ are shown in Table \ref{tbl:main}.

\begin{table}[htbp]
\centering
\begin{minipage}{0.135\textwidth}
    \centering
    \begin{tabular}{c|c}
        $p$ & $\#$ \\ \hline
        23 & 0 \\
        31 & 1 \\
        47 & 0 \\
        71 & 1 \\
        79 & 1 \\
        103 & 1 \\
        127 & 0 \\
        151 & 1 \\
        167 & 0 \\
        191 & 3 \\
        199 & 1 \\
        223 & 1 \\
        239 & 1 \\
        263 & 0 \\
        271 & 2 \\
        311 & 2 \\
        359 & 1 \\
        367 & 0 \\
        383 & 2 \\
        431 & 1 \\
        439 & 2 \\
        463 & 0 \\
        479 & 4 \\
        487 & 1 \\
        503 & 1 \\
        599 & 2 \\
        607 & 0 \\
        631 & 1 \\
        647 & 1 \\
        719 & 3 \\
        727 & 1 \\
        743 & 1 \\
        751 & 5 \\
        823 & 0 \\
        839 & 2 \\
        863 & 5 \\
        887 & 1 \\
        911 & 4 \\
        919 & 1 \\
        967 & 3 \\
        983 & 0 \\
        991 & 5 \\
        1031 & 1 \\
        1039 & 2
    \end{tabular}
\end{minipage}\hspace{-2mm}
\begin{minipage}{0.135\textwidth}
    \centering
    \begin{tabular}{c|c}
        $p$ & $\#$ \\ \hline
        1063 & 2 \\
        1087 & 0 \\
        1103 & 1 \\
        1151 & 6 \\
        1223 & 3 \\
        1231 & 5 \\
        1279 & 1 \\
        1303 & 0 \\
        1319 & 1 \\
        1327 & 2 \\
        1367 & 1 \\
        1399 & 1 \\
        1423 & 1 \\
        1439 & 1 \\
        1447 & 0 \\
        1471 & 3 \\
        1487 & 3 \\
        1511 & 1 \\
        1543 & 3 \\
        1559 & 1 \\
        1567 & 0 \\
        1583 & 2 \\
        1607 & 0 \\
        1663 & 2 \\
        1759 & 3 \\
        1783 & 0 \\
        1823 & 4 \\
        1831 & 2 \\
        1847 & 3 \\
        1871 & 5 \\
        1879 & 1 \\
        1951 & 4 \\
        1999 & 5 \\
        2039 & 2 \\
        2063 & 1 \\
        2087 & 1 \\
        2111 & 2 \\
        2143 & 5 \\
        2207 & 2 \\
        2239 & 2 \\
        2287 & 5 \\
        2311 & 1 \\
        2351 & 3 \\
        2383 & 4
    \end{tabular}
\end{minipage}\hspace{-2mm}
\begin{minipage}{0.135\textwidth}
    \centering
    \begin{tabular}{c|c}
        $p$ & $\#$ \\ \hline
        2399 & 3 \\
        2423 & 0 \\
        2447 & 2 \\
        2503 & 1 \\
        2543 & 1 \\
        2551 & 3 \\
        2591 & 2 \\
        2647 & 0 \\
        2663 & 2 \\
        2671 & 3 \\
        2687 & 1 \\
        2711 & 3 \\
        2719 & 4 \\
        2767 & 2 \\
        2791 & 2 \\
        2879 & 2 \\
        2887 & 0 \\
        2903 & 0 \\
        2927 & 0 \\
        2999 & 2 \\
        3023 & 2 \\
        3079 & 4 \\
        3119 & 3 \\
        3167 & 4 \\
        3191 & 2 \\
        3271 & 1 \\
        3319 & 2 \\
        3343 & 2 \\
        3359 & 6 \\
        3391 & 6 \\
        3407 & 5 \\
        3463 & 2 \\
        3511 & 5 \\
        3527 & 2 \\
        3559 & 1 \\
        3583 & 3 \\
        3607 & 2 \\
        3623 & 0 \\
        3631 & 2 \\
        3671 & 2 \\
        3719 & 4 \\
        3727 & 3 \\
        3767 & 1 \\
        3823 & 1
    \end{tabular}
\end{minipage}\hspace{-2mm}
\begin{minipage}{0.135\textwidth}
    \centering
    \begin{tabular}{c|c}
        $p$ & $\#$ \\ \hline
        3847 & 1 \\
        3863 & 2 \\
        3911 & 2 \\
        3919 & 2 \\
        3943 & 0 \\
        3967 & 5 \\
        4007 & 1 \\
        4079 & 6 \\
        4111 & 4 \\
        4127 & 2 \\
        4159 & 3 \\
        4231 & 1 \\
        4271 & 2 \\
        4327 & 1 \\
        4391 & 2 \\
        4423 & 2 \\
        4447 & 4 \\
        4463 & 3 \\
        4519 & 2 \\
        4567 & 1 \\
        4583 & 1 \\
        4591 & 2 \\
        4639 & 3 \\
        4663 & 1 \\
        4679 & 5 \\
        4703 & 6 \\
        4751 & 3 \\
        4759 & 3 \\
        4783 & 1 \\
        4799 & 3 \\
        4831 & 2 \\
        4871 & 2 \\
        4903 & 2 \\
        4919 & 2 \\
        4943 & 2 \\
        4951 & 4 \\
        4967 & 0 \\
        4999 & 1 \\
        5023 & 4 \\
        5039 & 1 \\
        5087 & 1 \\
        5119 & 3 \\
        5167 & 3 \\
        5231 & 7
    \end{tabular}
\end{minipage}\hspace{-2mm}
\begin{minipage}{0.135\textwidth}
    \centering
    \begin{tabular}{c|c}
        $p$ & $\#$ \\ \hline
        5279 & 6 \\
        5303 & 2 \\
        5351 & 3 \\
        5399 & 2 \\
        5407 & 4 \\
        5431 & 2 \\
        5471 & 3 \\
        5479 & 1 \\
        5503 & 0 \\
        5519 & 1 \\
        5527 & 2 \\
        5591 & 1 \\
        5623 & 1 \\
        5639 & 3 \\
        5647 & 3 \\
        5711 & 2 \\
        5743 & 1 \\
        5783 & 1 \\
        5791 & 4 \\
        5807 & 2 \\
        5839 & 4 \\
        5879 & 3 \\
        5903 & 7 \\
        5927 & 3 \\
        6007 & 0 \\
        6047 & 2 \\
        6079 & 5 \\
        6143 & 2 \\
        6151 & 2 \\
        6199 & 2 \\
        6247 & 2 \\
        6263 & 0 \\
        6271 & 4 \\
        6287 & 1 \\
        6311 & 2 \\
        6343 & 1 \\
        6359 & 4 \\
        6367 & 4 \\
        6551 & 4 \\
        6599 & 3 \\
        6607 & 0 \\
        6679 & 2 \\
        6703 & 1 \\
        6719 & 7
    \end{tabular}
\end{minipage}\hspace{-2mm}
\begin{minipage}{0.135\textwidth}
    \centering
    \begin{tabular}{c|c}
        $p$ & $\#$ \\ \hline
        6791 & 2 \\
        6823 & 3 \\
        6863 & 3 \\
        6871 & 2 \\
        6911 & 3 \\
        6959 & 4 \\
        6967 & 0 \\
        6983 & 2 \\
        6991 & 1 \\
        7039 & 2 \\
        7079 & 2 \\
        7103 & 3 \\
        7127 & 4 \\
        7151 & 3 \\
        7159 & 2 \\
        7207 & 1 \\
        7247 & 1 \\
        7351 & 1 \\
        7487 & 0 \\
        7559 & 1 \\
        7583 & 1 \\
        7591 & 2 \\
        7607 & 1 \\
        7639 & 1 \\
        7687 & 2 \\
        7703 & 0 \\
        7727 & 3 \\
        7759 & 5 \\
        7823 & 2 \\
        7879 & 4 \\
        7919 & 4 \\
        7927 & 1 \\
        7951 & 1 \\
        8039 & 4 \\
        8087 & 1 \\
        8111 & 4 \\
        8167 & 2 \\
        8191 & 4 \\
        8231 & 1 \\
        8263 & 1 \\
        8287 & 4 \\
        8311 & 3 \\
        8423 & 3 \\
        8431 & 3
    \end{tabular}
\end{minipage}\hspace{-2mm}
\begin{minipage}{0.135\textwidth}
    \centering
    \begin{tabular}{c|c}
        $p$ & $\#$ \\ \hline
        8447 & 4 \\
        8527 & 1 \\
        8543 & 2 \\
        8599 & 5 \\
        8623 & 3 \\
        8647 & 0 \\
        8663 & 1 \\
        8719 & 4 \\
        8783 & 3 \\
        8807 & 2 \\
        8831 & 1 \\
        8839 & 1 \\
        8863 & 2 \\
        8887 & 1 \\
        8951 & 3 \\
        8999 & 5 \\
        9007 & 1 \\
        9103 & 1 \\
        9127 & 0 \\
        9151 & 2 \\
        9199 & 3 \\
        9239 & 3 \\
        9311 & 5 \\
        9319 & 4 \\
        9343 & 4 \\
        9391 & 1 \\
        9431 & 4 \\
        9439 & 4 \\
        9463 & 2 \\
        9479 & 1 \\
        9511 & 2 \\
        9551 & 5 \\
        9623 & 4 \\
        9631 & 3 \\
        9679 & 6 \\
        9719 & 3 \\
        9743 & 4 \\
        9767 & 1 \\
        9791 & 5 \\
        9839 & 3 \\
        9871 & 2 \\
        9887 & 0 \\
        9967 & 3 \\
             & 
    \end{tabular}
\end{minipage}
\caption{The number $\#$ of isomorphism classes of superspecial plane quintic curves with cyclic automorphism group of order $8$ in characteristic $13 < p < 10000$ with $p \equiv 7 \pmod{8}$}\label{tbl:main}
\end{table}

\newpage 
From Table \ref{tbl:main}, we can observe that the number of isomorphism classes of superspecial plane quintic curves with automorphism groups $\cong \mathbb{Z}/8\mathbb{Z}$ seems to be irregular (recall from Theorem \ref{thm:main-10} that they are unlike those of superspecial plane quintic curves with automorphism groups $\cong \mathbb{Z}/10\mathbb{Z}$).
On the other hand, we found that there is no such a curve in the case where $p \not\equiv 7 \pmod{8}$ and $13 < p < 10000$. which completes the proof of Theorem \ref{thm:main-8}.\vspace{-1mm}

\section{Concluding remarks}\label{sec:concluding}
The results of Theorem \ref{thm:main-8} for the cases $p \equiv 1,3,5 \pmod{8}$ lead us to formulate the following expectation:\vspace{-1mm}
\begin{Ex}
If $p \not\equiv 7 \pmod{8}$, then there is no superspecial plane quintic curve with cyclic automorphism group of order $8$ in characteristic $p$.\vspace{-0.3mm}
\end{Ex}
\noindent This seems to arise from the fact that, when $p \not\equiv 7 \pmod{8}$, there are more equations that $\lambda$ must satisfy by Proposition \ref{prop:gauss_for_type8}, but a theoretical proof remains a topic for future work.
It is also an interesting problem to investigate  superspecial plane quintic curves with automorphism groups other than those in Table \ref{tbl:classification}.
In that case, we expect to apply multivariable hypergeometric series, as discussed in \cite{OH}.

\vspace{2.3mm}

\textsc{Graduate School of Information Science and Technology, The
University of Tokyo — 7-3-1 Hongo, Bunkyo-ku, Tokyo, 113-0033, Japan.}\par
\textit{E-mail address}: \url{ryo-ohashi@g.ecc.u-tokyo.ac.jp}

\vfill
\begin{thebibliography}{99}\vspace{-1mm}
\bibitem{BB} \textsc{E.\,Badr, F\hspace{-0.2mm}.\,Bars}: \textit{Automorphism groups of nonsingular plane curves of degree 5}, Communications in Algebra {\bf 44}, 4327--4340, 2016.\vspace{-1mm}
\bibitem{Brock} \textsc{B.\hspace{0.3mm}W\hspace{-0.4mm}.\,Brock}: \textit{Superspecial curves of genera two and three}, Thesis (Ph.D.)-Princeton \hspace{-0.2mm}University, 1993.\vspace{-1mm}
\bibitem{Dolgachev} \textsc{I.\hspace{0.4mm}V\hspace{-0.3mm}.\,Dolgachev}: \textit{Classical Algebraic Geometry: \hspace{-0.2mm}A \hspace{-0.2mm}Modern \hspace{-0.2mm}View}, \hspace{0.2mm}Cambridge University Press, 2012.\vspace{-1mm}
\bibitem{Ekedahl} \textsc{T.\,Ekedahl}: \textit{On supersingular curves and abelian varieties}, Math.\,Scand.\,{\bf 60}, 151--178, 1987.\vspace{-1mm}
\bibitem{Hartshorne} \textsc{R.\,Hartshorne}: \textit{Algebraic Geometry}, GTM\hspace{0.7mm}{\bf 52}, Springer--Verlag, 1977.\vspace{-1mm}
% \bibitem{Hashimoto} \textsc{K\hspace{-0.2mm}.\,Hashimoto}: \textit{Class numbers of positive definite ternary quaternion hermitian forms}, \hspace{0.2mm}Proc.\,Jpn.\,Acad., Ser.\,A \hspace{-0.3mm}{\bf 59}, 490-4403, 1983.\vspace{-1mm}
% \bibitem{Igusa} \textsc{J.\,Igusa}: \textit{Arithmetic variety of moduli for genus two}, Ann.\,Math.\hspace{0.5mm}{\bf 72}(2), 612--649, 1960.\vspace{-1mm}
\bibitem{IKO} \textsc{T\hspace{-0.2mm}.\,Ibukiyama, \hspace{-0.2mm}T\hspace{-0.2mm}.\,Katsura, F.\,Oort}: \textit{Supersingular curves of genus two and class numbers}, Compos. Math.\,{\bf 57}(2), 127--152, 1986.\vspace{-1mm}
\bibitem{Ishii} \textsc{N.\,Ishii}: \textit{The automorphism group and invariants of a curve of genus 2}, J.\hspace{0.8mm}Res.\hspace{0.8mm}Inst.\hspace{0.8mm}Sci.\hspace{0.8mm}Tech., Nihon Univ.\,{\bf 126}, 6--11, 2011.\vspace{-1mm}
\bibitem{KNT} \textsc{A\hspace{-0.2mm}.\hspace{0.6mm}Kazemifard, A.\hspace{0.5mm}R.\hspace{0.5mm}Naghipour, S.\hspace{0.3mm}Tafazolian}: \textit{A note on superspecial and maximal curves}, Bull. Iranian Math.\,Soc.\,{\bf 39}, 405--413, 2013.\vspace{-1mm}
\bibitem{KW} \textsc{T\hspace{-0.2mm}.\,Kodama, \hspace{-0.2mm}T\hspace{-0.2mm}.\hspace{0.5mm}Washio}: \textit{Hasse-Witt matrices of Fermat curves}, Manuscripta Math.\hspace{0.5mm}{\bf 60}, 185--195, 1988.\vspace{-1mm}
% \bibitem{LGRS} \textsc{D.\,Lombardo, \hspace{-0.2mm}E.\,L.\,García, \hspace{-0.2mm}C.\hspace{0.6mm}Ritzenthaler, \hspace{-0.2mm}J\hspace{-0.2mm}.\,Sijsling}: \textit{Decomposing Jacobians via Galois covers}, Exp.\,Math.\,{\bf 32}, 218--240, 2023.\vspace{-1mm}
% \bibitem{MT} \textsc{S.\,Meagher, J.\,Top}: \textit{Twists of genus three curves over finite fields}, Finite Fields \hspace{-0.1mm}Appl.\,{\bf 16}(5), 347--368, 2010.\vspace{-1mm}
\bibitem{MS} \textsc{M\hspace{-0.1mm}.\,Montanucci, P\hspace{-0.4mm}.\,Speziali}: \hspace{-0.3mm}\textit{The $a$-numbers of Fermat and Hurwitz curves}, J\hspace{-0.2mm}.\,Pure \hspace{-0.2mm}Appl.\,Algebra\,{\bf 222}, 477--488, 2018.\vspace{-1mm}
\bibitem{MK} \textsc{T\hspace{-0.3mm}.\,Moriya, M\hspace{-0.1mm}.\,Kudo}: \textit{Some explicit arithmetic on curves of genus three and their applications}, preprint, arXiv:\,2209.02926.\vspace{-1mm}
\bibitem{Ohashi} \textsc{R.\,Ohashi}: \textit{On the Rosenhain forms of superspecial curves of genus two}, Finite \hspace{-0.1mm}Fields \hspace{-0.2mm}Appl.\,{\bf 97}, 102445, 2024.\vspace{-1mm}
\bibitem{OH} \textsc{R.\,Ohashi, S.\,Harashita}: \textit{Differential forms on the curves associated to Appell-Lauricella hypergeome-\\tric series and the Cartier operator on them}, Yokohama Mathematical Journal\hspace{1mm}{\bf 69}, 1--32, 2023.\vspace{-1mm}
\bibitem{OK} \textsc{R.\,Ohashi, M\hspace{-0.2mm}.\,Kudo}: \textit{Computing superspecial hyperelliptic curves of genus 4 with automorphism group properly containing the Klein 4-group}, J.\,Comput.\,Algebra\hspace{1mm}{\bf 11}, 100020, 2024.\vspace{-1mm}
\bibitem{OKH} \textsc{R.\,Ohashi, M\hspace{-0.2mm}.\,Kudo, S.\,Harashita}: \textit{The $a$-numbers of non-hyperelliptic curves of genus 3 with cyclic automorphism group of order $6$}, Acta Arith.\,{\bf 216}(3), 227--248, 2024.\vspace{-1mm}
\bibitem{Oort} \textsc{F\hspace{-0.2mm}.\,Oort}: \textit{Hyperelliptic supersingular curves}, Prog.\hspace{0.8mm}Math.\,{\bf 89}, 247--284, 1991.\vspace{-1mm}
% \bibitem{TT14} \textsc{S.\,Tafazolian, F\hspace{-0.3mm}.\,Torres}: \textit{On the curve $y^n = x^m+x$ over finite fields}, J.\hspace{0.7mm}Number \hspace{-0.2mm}Theory\hspace{0.8mm}{\bf 145}, 51--66, 2014.
% \bibitem{Yui} \textsc{N\hspace{-0.1mm}.\hspace{0.3mm}Yui}: \textit{On the Jacobian varieties of hyperelliptic curves over fields of characteristic $p \hspace{-0.2mm}>\hspace{-0.2mm} 2$}, J.\,Algebra\,{\bf 52}, 378--410, 1978.
\end{thebibliography}
\end{document}